\documentclass[10pt]{article}
\usepackage{amsmath}
\usepackage{amssymb}
\usepackage{amsthm}
\usepackage{amscd}
\usepackage{xypic}
\usepackage{delarray}
\usepackage{mathrsfs}

\headsep=1cm \numberwithin{equation}{section}
\def\C{{\mathbb C}}

\def\F{{\mathbb F}}

\def\Q{{\mathbb Q}}
\def\R{{\mathbb R}}

\def\Z{{\mathbb Z}}

\newtheorem{theorem}{Theorem}[section]
\newtheorem{lemma}[theorem]{Lemma}
\newtheorem{proposition}[theorem]{Proposition}
\newtheorem{corollary}[theorem]{Corollary}
\newtheorem{definition and lemma}[theorem]{Definition and Lemma}

\theoremstyle{definition}
\newtheorem{definition}[theorem]{Definition}
\newtheorem{example}[theorem]{Example}

\theoremstyle{remark}
\newtheorem{remark}[theorem]{Remark}

\numberwithin{equation}{section}

\date{}

\begin{document}
\author{WonTae Hwang}

\title{On a classification of the automorphism groups of polarized abelian surfaces over finite fields}

\maketitle

\begin{abstract}
We give a classification of maximal elements of the set of finite groups that can be realized as the full automorphism groups of polarized abelian surfaces over finite fields.
\end{abstract}

\section{Introduction}
\qquad Let $k$ be a field, and let $X$ be an abelian surface over $k.$ It is well-known (see \cite[\S19]{8}) that $X$ is either simple or isogenous to a product of two elliptic curves over $k$. We denote the endomorphism ring of $X$ over $k$ by $\textrm{End}_k(X)$. It is a free $\Z$-module of rank $\leq 16.$ We also let $\textrm{End}^0_k(X)=\textrm{End}_k(X) \otimes_{\Z} \Q.$ This $\Q$-algebra $\textrm{End}_k^0(X)$ is called the endomorphism algebra of $X$ over $k.$ Then $\textrm{End}_k^0(X)$ is a finite dimensional semisimple algebra over $\Q$ with $\textrm{dim}_{\Q} \textrm{End}_k^0(X) \leq 16.$ Moreover, if $X$ is simple, then $\textrm{End}_k^0(X)$ is a division algebra over $\Q$. Now, it is a well-known fact that $\textrm{End}_k (X)$ is a $\Z$-order in $\textrm{End}_k^0(X).$ The group $\textrm{Aut}_k(X)$ of the automorphisms of $X$ over $k$ is not finite, in general. But if we fix a polarization $\mathcal{L}$ on $X$, then the group $\textrm{Aut}_k(X,\mathcal{L})$ of the automorphisms of the polarized abelian surface $(X,\mathcal{L})$ is always finite. The goal of this paper is to classify all finite groups that can be the automorphism group $\textrm{Aut}_k(X,\mathcal{L})$ of a polarized abelian surface $(X,\mathcal{L})$ over a finite field $k,$ which are maximal in the sense of Definition \ref{def 20} below. \\

Along this line, for the case when $k=\C,$ Fujiki \cite{4} gave a complete list of all pairs $(X,G)$ where $X$ is a complex torus of dimension $2$, and $G$ is a finite subgroup of $\textrm{Aut}(X)$, up to conjugacy. Also, Birkenhake and Lange~\cite{2} provided a classification of a finite subgroup $G$ of $\textrm{Aut}(X)$ that is maximal in the isogeny class of $X$ where $X$ is an abelian surface over $k=\C$. On the other hand, for the case when $\textrm{char}~k = p>0$, it seems that not much is known. \\

Our main result is summarized in the following:
\begin{theorem}\label{main theorem}
  The possibilities for maximal automorphism groups of a polarized abelian surface $(X, \mathcal{L})$ over a finite field are given by Tables 10, 11, 12, and 13 in the case of $X$ being simple, a product of two non-isogenous elliptic curves, a power of an ordinary elliptic curve, and a power of a supersingular elliptic curve, respectively.
\end{theorem}
For more details, see Theorems \ref{thm old 24}, \ref{prodnonisoellip}, \ref{powordelli}, and \ref{pow of supell} below. \\

This paper is organized as follows: In Section~\ref{prelim}, we introduce several facts which are related to our desired classification. Explicitly, we will recall some facts about endomorphism algebras of abelian varieties ($\S$\ref{end alg av}), the theorem of Tate ($\S$\ref{thm Tate sec}), Honda-Tate theory ($\S$\ref{thm Honda}), a result of Waterhouse ($\S$\ref{thm waterhouse}), and maximal orders over a Dedekind domain ($\S$\ref{max ord dede}). In Section~\ref{aut gps ell}, we give a classification of the automorphism groups of elliptic curves over finite fields. In Section~\ref{quat mat rep}, we record some useful results about quaternionic matrix groups based on a paper of Nebe \cite{10}. In Section~\ref{findiv}, we find all the finite groups that can be embedded in certain division algebras using a result of Amitsur \cite{1}. In Section~\ref{main}, we finally obtain the desired classification using the facts that were introduced in the previous sections. \\

In the sequel, let $q=p^a$ for some prime number $p$ and an integer $a \geq 1,$ unless otherwise stated.

\section{Preliminaries}\label{prelim}
\qquad In this section, we recall some of the facts in the general theory of abelian varieties over a field and maximal orders over a Dedekind domain. Our main references are \cite{3}, \cite{8}, and \cite{11}.

\subsection{Endomorphism algebras of abelian varieties}\label{end alg av}
\qquad In this section, we give some basic facts about the endomorphism algebra of an abelian variety over a field, and recall Albert's classification regarding the endomorphism algebra of a simple abelian variety. Throughout this section, let $k$ be a field. \\

Let $X$ be an abelian variety over $k.$ Then the set $\textrm{End}_k (X)$ of endomorphisms of $X$ over $k$ has a natural ring structure. If we want to work with $\Q$-algebras, then we define $\textrm{End}_k^0(X):=\textrm{End}_k (X) \otimes_{\Z} \Q.$ The $\Q$-algebra $\textrm{End}_k^0 (X)$ is called the \emph{endomorphism algebra} of $X.$ \\

If $X$ is a simple abelian variety over $k,$ then $\textrm{End}_k^0(X)$ is a division algebra over $\Q.$ If $X$ is any abelian variety over $k,$ then it is well-known that there exist simple abelian varieties $Y_1,\cdots,Y_n$ over $k,$ no two of which are $k$-isogenous, and positive integers $m_1,\cdots,m_n$ such that $X$ is $k$-isogenous to $Y_1^{m_1} \times \cdots \times Y_n^{m_n}.$ In this situation, we have
\begin{equation*}
\textrm{End}_k^0 (X) \cong M_{m_1}(D_1) \times \cdots \times M_{m_n}(D_n)
\end{equation*}
where $D_i :=\textrm{End}_k^0 (X_i)$ for each $i.$ Moreover, it can be shown that $\textrm{End}_k^0 (X)$ is a finite dimensional semisimple $\Q$-algebra of dimension at most $4 \cdot (\textrm{dim}~X)^2.$ \\

We conclude this section by recalling Albert's classification. Let $X$ be a simple abelian variety over $k$, and we choose a polarization $\lambda : X \rightarrow \widehat{X}$. Using the polarization $\lambda,$ we can define an involution, called the \emph{Rosati involution}, $^{\vee}$ on $\textrm{End}_k^0(X).$ (For a more detailed discussion about the Rosati involution, see \cite[\S20]{8}.) In this way, to the pair $(X,\lambda)$ we associate the pair $(D, ^{\vee})$ with $D=\textrm{End}_k^0(X)$ and $^{\vee}$, the Rosati involution on $D$. We know that $D$ is a simple division algebra over $\Q$ of finite dimension and that $^{\vee}$ is a positive involution. Let $K$ be the center of $D$ so that $D$ is a central simple $K$-algebra, and let $K_0 = \{ x \in K~|~x^{\vee} = x \}$ be the subfield of symmetric elements in $K.$ We know that either $K_0 = K$, in which case, $^{\vee}$ is said to be of the first kind, or that $K$ is a quadratic extension of $K_0$, in which case, $^{\vee}$ is said to be of the second kind. By a theorem of Albert, the pair $(D,^{\vee})$ is of one of the following four types: \\

(i) Type I: $K_0 = K=D$ is a totally real field and $^{\vee} = \textrm{id}_D.$ \\

(ii) Type II: $K_0 = K$ is a totally real field, and $D$ is a quaternion algebra over $K$ with $D\otimes_{K,\sigma} \R \cong M_2(\R)$ for every embedding $\sigma : K \hookrightarrow \R.$ Now, let $a \mapsto \textrm{Trd}_{D/K}(a)-a$ be the canonical involution on $D.$ Then there is an element $b \in D$ such that $b^2 \in K$ is totally negative, and such that $a^{\vee}=b (\textrm{Trd}_{D/K}(a)-a) b^{-1}$ for all $a \in D.$ We have an isomorphism $D \otimes_{\Q}\R \cong \prod_{\sigma : K \hookrightarrow \R} M_2(\R)$ such that the involution $^{\vee}$ on $D \otimes_{\Q} \R$ corresponds to the involution $(A_1, \cdots, A_e) \mapsto (A_1^t , \cdots, A_e^t)$ where $e=[K:\Q].$ \\

(iii) Type III: $K_0 = K$ is a totally real field, and $D$ is a quaternion algebra over $K$ with $D\otimes_{K,\sigma} \R \cong \mathbb{H}$ for every embedding $\sigma : K \hookrightarrow \R$ (where $\mathbb{H}$ is the Hamiltonian quaternion algebra over $\R$). Now, let $a \mapsto \textrm{Trd}_{D/K}(a)-a$ be the canonical involution on $D.$ Then $^{\vee}$ is equal to the canonical involution on $D.$ We have an isomorphism $D \otimes_{\Q}\R \cong \prod_{\sigma : K \hookrightarrow \R} \mathbb{H}$ such that the involution $^{\vee}$ on $D \otimes_{\Q} \R$ corresponds to the involution $(\alpha_1, \cdots, \alpha_e) \mapsto (\overline{\alpha_1} , \cdots, \overline{\alpha_e})$ where $e=[K:\Q].$ \\

(iv) Type IV: $K_0 $ is a totally real field, $K$ is a totally imaginary quadratic field extension of $K_0$. Write $a \mapsto \overline{a}$ for the unique non-trivial automorphism of $K$ over $K_0$; this automorphism is usually referred to as complex conjugation. If $\nu$ is a finite place of $K$, write $\overline{\nu}$ for its complex conjugate. The algebra $D$ is a central simple algebra over $K$ such that: (a) If $\nu$ is a finite place of $K$ with $\nu = \overline{\nu},$ then $\textrm{inv}_{\nu}(D)=0$; (b) For any place $\nu$ of $K$, we have $\textrm{inv}_{\nu}(D)+\textrm{inv}_{\overline{\nu}}(D)=0$ in $\Q/\Z.$ If $d$ is the degree of $D$ as a central simple $K$-algebra, then we have an isomorphism $D \otimes_{\Q}\R \cong \prod_{\sigma : K_0 \hookrightarrow \R} M_d(\C)$ such that the involution $^{\vee}$ on $D \otimes_{\Q} \R$ corresponds to the involution $(A_1, \cdots, A_{e_0}) \mapsto (\overline{A_1}^t , \cdots, \overline{A_{e_0}}^t)$ where $e_0=[K_0:\Q].$ \\

Keeping the notations as above, we let
\begin{equation*}
  e_0= [K_0 :\Q],~~~e=[K:\Q],~~~\textrm{and}~~~d=[D:K]^{\frac{1}{2}}.
\end{equation*}

As our last preliminary fact of this section, we impose some numerical restrictions on those values $e_0, e,$ and $m$ in the next table, following \cite[\S21]{8}.
\begin{center}
  \begin{tabular}{|c|c|c|}
\hline
$$ & $\textrm{char}(k)=0$ & $\textrm{char}(k)=p>0$ \\
\hline
$\textrm{Type I}$ & $e|g$ & $e|g$  \\
\hline
$\textrm{Type II}$ & $2e|g$  & $2e|g$ \\
\hline
$\textrm{Type III}$ & $2e|g$ & $e|g$ \\
\hline
$\textrm{Type IV}$ & $e_0 d^2 |g$ &$ e_0 d |g $ \\
\hline
\end{tabular}
\vskip 4pt
\textnormal{Table 1}
\end{center}

\subsection{The theorem of Tate}\label{thm Tate sec}
\qquad In this section, we recall an important theorem of Tate, and give some interesting consequences of it. \\

Let $k$ be a field and let $l$ be a prime number with $l \ne \textrm{char}(k)$. If $X$ is an abelian variety of dimension $g$ over $k,$ then we can introduce the Tate $l$-module $T_l X$ and the corresponding $\Q_l$-vector space $V_l X :=T_l X \otimes_{\Z_l} \Q_l.$ It is well-known that $T_l X$ is a free $\Z_l$-module of rank $2g$ and $V_l X$ is a $2g$-dimensional $\Q_l$-vector space. In \cite{13}, Tate showed the following important result:

\begin{theorem}\label{thm Tate}
Let $k$ be a finite field, $\overline{k}$ an algebraic closure of $k,$ and let $\Gamma = \textrm{Gal}(\overline{k}/k).$ If $l$ is a prime number with $l \ne \textrm{char}(k),$ then we have: \\
(a) For any abelian variety $X$ over $k,$ the representation
\begin{equation*}
 \rho_l =\rho_{l,X} : \Gamma \rightarrow \textrm{GL}(V_l X)
\end{equation*}
is semisimple. \\
(b) For any two abelian varieties $X$ and $Y$ over $k,$ the map
\begin{equation*}
 \Z_l \otimes_{\Z} \textrm{Hom}_k(X,Y) \rightarrow \textrm{Hom}_{\Gamma}(T_l X, T_l Y)
\end{equation*}
is an isomorphism.
\end{theorem}

Now, we recall that an abelian variety $X$ over a finite field $k$ is called \emph{elementary} if $X$ is $k$-isogenous to a power of a simple abelian variety over $k.$ Then, as an interesting consequence of Theorem \ref{thm Tate}, we have the following
\begin{corollary}\label{cor TateEnd0}
  Let $X$ be an abelian variety of dimension $g$ over a finite field $k.$ Then we have:\\
  (a) The center $Z$ of $\textrm{End}_k^0(X)$ is the subalgebra $\Q[\pi_X].$ In particular, $X$ is elementary if and only if $\Q[\pi_X]=\Q(\pi_X)$ is a field, and this occurs if and only if $f_X$ is a power of an irreducible polynomial in $\Q[t]$ where $f_X$ denotes the characteristic polynomial of $\pi_X.$ \\
 (b) Suppose that $X$ is elementary. Let $h=f_{\Q}^{\pi_X}$ be the minimal polynomial of $\pi_X$ over $\Q$. Further, let $d=[\textrm{End}_k^0(X):\Q(\pi_X)]^{\frac{1}{2}}$ and $e=[\Q(\pi_X):\Q].$ Then $de =2g$ and $f_X = h^d.$ \\
(c) We have $2g \leq \textrm{dim}_{\Q} \textrm{End}^0_k (X) \leq (2g)^2$ and $X$ is of CM-type. \\
(d) The following conditions are equivalent: \\
  \indent (d-1) $\textrm{dim}_{\Q} \textrm{End}_k^0(X)=2g$; \\
  \indent (d-2) $\textrm{End}_k^0(X)=\Q[\pi_X]$; \\
  \indent (d-3) $\textrm{End}_k^0(X)$ is commutative; \\
  \indent (d-4) $f_X$ has no multiple root. \\
(e) The following conditions are equivalent: \\
  \indent (e-1) $\textrm{dim}_{\Q} \textrm{End}_k^0(X)=(2g)^2$; \\
  \indent (e-2) $\Q[\pi_X]=\Q$; \\
  \indent (e-3) $f_X$ is a power of a linear polynomial; \\
  \indent (e-4) $\textrm{End}^0_k(X) \cong M_g(D_{p,\infty})$ where $D_{p,\infty}$ is the unique quaternion algebra over $\Q$ that is ramified at $p$ and $\infty$, and split at all other primes; \\
  \indent (e-5) $X$ is supersingular with $\textrm{End}_k(X) = \textrm{End}_{\overline{k}}(X_{\overline{k}})$; \\
  \indent (e-6) $X$ is isogenous to $E^g$ for a supersingular elliptic curve $E$ over $k$ all of whose endomorphisms are defined over $k.$
\end{corollary}
\begin{proof}
  For a proof, see \cite[Theorem 2]{13}. 
\end{proof}

For a precise description of the structure of the endomorphism algebra of an elementary abelian variety $X$, viewed as a simple algebra over its center $\Q[\pi_X]$, we record the following two results:
\begin{proposition}\label{local inv}
  Let $X$ be an elementary abelian variety over a finite field $k=\F_q.$ Let $K=\Q[\pi_X].$ If $\nu$ is a place of $K$, then the local invariant of $\textrm{End}_k^0(X)$ in the Brauer group $\textrm{Br}(K_{\nu})$ is given by
  \begin{equation*}
    \textrm{inv}_{\nu}(\textrm{End}_k^0(X))=\begin{cases} 0 & \mbox{if $\nu$ is a finite place not above $p$}; \\ \frac{\textrm{ord}_{\nu}(\pi_X)}{\textrm{ord}_{\nu}(q)} \cdot [K_{\nu}:\Q_p] & \mbox{if $\nu$ is a place above $p$}; \\ \frac{1}{2} & \mbox{if $\nu$ is a real place of $K$}; \\ 0 & \mbox{if $\nu$ is a complex place of $K$}. \end{cases}
  \end{equation*}
\end{proposition}
\begin{proof}
  For a proof, see \cite[Corollary 16.30]{3}.
\end{proof}

\begin{proposition}\label{index end alg}
  Let $X$ be a simple abelian variety over a finite field $k.$ Let $d$ be the degree of the division algebra $D:=\textrm{End}_k^0(X)$ over its center $\Q(\pi_X)$ (so that $d=[D:\Q(\pi_X)]^{\frac{1}{2}}$ and $f_X = (f_{\Q}^{\pi_X})^d$). Then $d$ is the least common denominator of the local invariants $\textrm{inv}_{\nu}(D).$
\end{proposition}
\begin{proof}
  For a proof, see \cite[Corollary 16.32]{3}.
\end{proof}

\begin{remark}\label{Qpi Real}
  Let $X$ be a simple abelian variety over $k=\F_q$. Since $X$ is of CM-type by Corollary \ref{cor TateEnd0}, we have that $D:=\textrm{End}_k^0(X)$ is either of Type III or of Type IV in the Albert's classification. We can see in which case we are by looking at the center $K:=\Q[\pi_X].$ In fact, it can be shown that either $K$ is totally real ($D$ is of Type III) or $K$ is a CM-field ($D$ is of Type IV). \\

  The first case is very exceptional. In fact, let $h=f_{\Q}^{\pi_X},$ the minimal polynomial of $\pi_X$ over $\Q.$ As in Proposition \ref{index end alg}, we have $f_X = h^d$ where $d$ is the degree of $D$ over $K.$ If $K$ is totally real, then all complex zeros of $f_X$ are real numbers of absolute value $\sqrt{q}.$ We distinguish two cases: \\

  (1) If $q$ is a square, then $h=t \pm \sqrt{q}$ so that $K=\Q$ and $d=\frac{2g}{e}=2g$ by Corollary \ref{cor TateEnd0}-(b). By Proposition \ref{local inv}, we have $\textrm{inv}_p (D)=\frac{1}{2}=\textrm{inv}_{\infty}(D)$ and $\textrm{inv}_l (D)=0$ at all other places $l$ of $K.$ Then it follows from Proposition \ref{index end alg} that $d=2.$ Hence, we get $g=1$, and $X$ is an elliptic curve over $k$ and $D$ is the unique quaternion algebra over $\Q$ that is ramified at $p$ and $\infty$, and split at all other primes. This quaternion algebra over $\Q$ will be denoted by $D_{p,\infty}$ in the sequel. \\

  (2) If $q$ is not a square, then $h=t^2 -q$ so that $K=\Q[\sqrt{q}]=\Q[\sqrt{p}]$ and $d=\frac{2g}{e}=g$ by Corollary \ref{cor TateEnd0}-(b). In this case, there is a unique prime $\mathfrak{p}$ of $K$ above $p$ and we have $\textrm{inv}_{\mathfrak{p}}(D)=0.$ Hence, $D$ is the unique quaternion algebra over $K$ that is ramified at the two infinite places of $K$ and split at all finite primes. By Proposition \ref{index end alg}, it follows that $d=2$. Hence, we get $g=2$, and $X$ is a simple abelian surface over $k.$ If we extend scalars from $k$ to its quadratic extension $k^{\prime}:=\F_{q^2}$, then $X^{\prime}:=X \times_k k^{\prime}$ has characteristic polynomial $f_{X^{\prime}}(t)=(t-q)^4.$ In this case, we have that $\textrm{End}_{k^{\prime}}^0(X^{\prime})\cong M_2(D_{p,\infty})$ and $X^{\prime}$ is isogenous to $E^2$ where $E$ is a supersingular elliptic curve over $k^{\prime}$ with characteristic polynomial $f_E(t)=(t-q)^2.$ \\

  Except in these particular cases, the center $K=\Q[\pi_X]$ is always a CM-field. Note also that in case (2) above, the structure of the endomorphism algebra of $X$ changes when we extend scalars from $\F_q$ to $\F_{q^2}.$
\end{remark}

\subsection{Abelian varieties up to isogeny and Weil numbers: Honda-Tate theory}\label{thm Honda}
\qquad In this section, we recall an important theorem of Honda and Tate. To achieve our goal, we first give the following
\begin{definition}\label{qWeil Def}
  (a) A \emph{$q$-Weil number} is an algebraic integer $\pi$ such that $| \iota(\pi) | = \sqrt{q}$ for all embeddings $\iota : \Q[\pi] \hookrightarrow \C.$ \\
  (b) Two $q$-Weil numbers $\pi$ and $\pi^{\prime}$ are said to be \emph{conjugate} if they have the same minimal polynomial over $\Q,$ or equivalently, there is an isomorphism $\Q[\pi] \rightarrow \Q[\pi^{\prime}]$ sending $\pi$ to $\pi^{\prime}.$
\end{definition}

Regarding $q$-Weil numbers, the following facts are well-known:
\begin{remark}\label{qWeil rem}
Let $X$ and $Y$ be abelian varieties over a finite field $k=\F_q.$ Then we have: \\
(1) The Frobenius endomorphism $\pi_X$ is a $q$-Weil number. \\
(2) $X$ and $Y$ are $k$-isogenous if and only if $\pi_X$ and $\pi_Y$ are conjugate.
\end{remark}

Now, we introduce our main result of this section:

\begin{theorem}\label{thm HondaTata}
 For every $q$-Weil number $\pi$, there exists a simple abelian variety $X$ over $\F_q$ such that $\pi_X$ is conjugate to $\pi$, where $\pi_X$ denotes the Frobenius endomorphism of $X.$ Moreover, we have a bijection between the set of isogeny classes of simple abelian varieties over $\F_q$ and the set of conjugacy classes of $q$-Weil numbers given by $X \mapsto \pi_X$.
\end{theorem}
The inverse of the map $X \mapsto \pi_X$ associates to a $q$-Weil number $\pi$ a simple abelian variety $X$ such that $f_X$ is a power of the minimal polynomial $f_{\Q}^{\pi}$ of $\pi$ over $\Q.$
\begin{proof}
For a proof, see \cite[Main Theorem]{5} or \cite[\S16.5]{3}.
\end{proof}

\subsection{Isomorphism classes contained in an isogeny class}\label{thm waterhouse}
\qquad In this section, we will give a useful result of Waterhouse~\cite{15}. Throughout this section, let $k=\F_q.$ \\

Let $X$ be an abelian variety over $k.$ Then $\textrm{End}_k(X)$ is a $\Z$-order in $\textrm{End}_k^0(X)$ containing $\pi_X$ and $q/\pi_X.$ If a ring is the endomorphism ring of an abelian variety, then we may consider a left ideal of the ring, and give the following
\begin{definition}\label{Roccur}
  Let $X$ be an abelian variety over $k$ with $R:=\textrm{End}_k(X),$ and let $I$ be a left ideal of $R$. \\
  (a) We define $H(X,I)$ to be the intersection of the kernels of all elements of $I$. This is a finite subgroup scheme of $X.$ \\
  (b) We define $X_I$ to be the quotient of $X$ by $H(X,I)$ i.e.\ $X_I=X/H(X,I).$ This is an abelian variety over $k$ that is $k$-isogenous to $X.$
\end{definition}

Now, to introduce our main result of this section, we need the following
\begin{lemma}\label{endchange}
  Let $X$ be an abelian variety over $k$ with $R:=\textrm{End}_k(X)$, and let $I$ be a left ideal of $R.$ Then $\textrm{End}_k(X_I)$ contains $O_r (I):=\{x \in \textrm{End}_k^0(X)~|~I x \subseteq I\}$, the right order of $I,$ and equals it if $I$ is a kernel ideal.
\end{lemma}
\begin{proof}
For a proof, see \cite[Lemma 16.56]{3} or \cite[Proposition 3.9]{15}.
\end{proof}

Using this lemma, we can prove the following result, which plays an important role:
\begin{proposition}\label{endposs}
  Let $X$ be an abelian variety over $k$. Then every maximal order in $D:=\textrm{End}_k^0(X)$ occurs as the endomorphism ring of an abelian variety in the isogeny class of $X.$
\end{proposition}
\begin{proof}
For a proof, see \cite[Theorem 3.13]{15}.
\end{proof}

\subsection{Maximal orders over a Dedekind domain}\label{max ord dede} 
\qquad In this section, we recall the general theory of maximal orders over a Dedekind domain that will be used later in this paper. \\

Throughout this section, let $R$ be a noetherian integral domain with the quotient field $K,$ and let $A$ be a finite dimensional $K$-algebra. Recall that a maximal $R$-order in $A$ is an $R$-order which is not properly contained in any other $R$-order in $A.$ For our later use, we introduce several results about maximal orders:
\begin{theorem}\label{mat max}
  Let $A$ be a finite dimensional $K$-algebra. If $\Lambda$ is a maximal $R$-order in $A,$ then for each $n \geq 1,$ $M_n(\Lambda)$ is a maximal $R$-order in $M_n(A).$ If $R$ is integrally closed, then $M_n(R)$ is a maximal $R$-order in $M_n(K).$
\end{theorem}
\begin{proof}
 For a proof, see \cite[Theorem 8.7]{11}.
\end{proof}

If we impose conditions that $R$ is integrally closed and $A$ is a separable $K$-algebra, then we can have the following result saying that a decomposition of the $K$-algebra $A$ into simple components yields a corresponding decomposition of maximal orders in $A$:
\begin{theorem}\label{max gen}
  Let $A$ be a separable $K$-algebra with simple components $\{A_i\}_{1 \leq i \leq t}$ and let $R_i$ be the integral closure of $R$ in the center $K_i$ of $A_i$ for each $i.$ Then we have: \\
  (a) For each maximal $R$-order $\Lambda$ in $A,$ we have $\Lambda = \sum^{\bullet} \Lambda e_i$ where $\{e_i\}_{1 \leq i \leq t}$ are the central idempotents of $A$ such that $A_i = A e_i$ for each $i.$ Moreover, each $\Lambda e_i$ is a maximal $R$-order in $A_i=A e_i.$ \\
  (b) If $\Lambda_i $ is a maximal $R$-order in $A_i$ for each $i,$ then $\sum^{\bullet} \Lambda_i$ is a maximal $R$-order in $A.$ \\
  (c) An $R$-order $\Lambda_i $ in $A_i$ is a maximal $R$-order if and only if $\Lambda_i$ is a maximal $R_i$-order in $A_i.$
\end{theorem}
\begin{proof}
  For a proof, see \cite[Theorem 10.5]{11}.
\end{proof}

Finally, we further assume that $R$ is a Dedekind domain with its quotient field $K \ne R$. Let $A$ be a separable $K$-algebra (which is simple). In this last theorem, we determine all maximal $R$-orders in $A.$
\begin{theorem}\label{mat max 2}
  Let $A=\textrm{Hom}_{D}(V,V) \cong M_r(D)$ be a simple algebra, where $V$ is a right vector space of dimension $r$ over a division algebra $D$ with center $K$. Let $\Delta$ be a fixed maximal $R$-order in $D$, and let $M$ be a full right $\Delta$-lattce in $V.$ Then $\textrm{Hom}_{\Delta}(M,M)$ is a maximal $R$-order in $A.$ If $\Lambda^{\prime}$ is a maximal $R$-order in $A,$ then there is a full right $\Delta$-lattice $N$ in $V$ such that $\Lambda^{\prime}=\textrm{Hom}_{\Delta}(N,N).$
\end{theorem}

\begin{proof}
 For a proof, see \cite[Theorem 21.6]{11}.
\end{proof}

\section{Automorphism groups of elliptic curves}\label{aut gps ell}
\qquad To consider the case of an abelian surface that is isogenous to a product of two elliptic curves, it is natural to discuss the automorphism groups of elliptic curves. First of all, we recall a classical result of Deuring, following \cite{15}. 

\begin{definition}\label{supsing ell}
  An elliptic curve $E$ over a finite field $k$ is called \emph{supersingular} if $\textrm{End}_{\overline{k}}(E)$ is non-commutative.
\end{definition}

To introduce the desired result, we need:
\begin{lemma}\label{prime dec}
  Let $\beta$ be an integer such that $|\beta|< 2 \sqrt{q}.$ In $K:= \Q(\sqrt{\beta^2 - 4q})$, we have: \\
  (a) $p$ ramifies if \\
  \indent (i) $\beta=0$ and $a$ is odd; \\
  \indent (ii) $\beta=0$, $a$ is even, and $p=2$; \\
  \indent (iii) $\beta=\pm \sqrt{q}$, $a$ is even, and $p=3$; \\
  \indent (iv) $\beta= \pm p^{\frac{a+1}{2}},$ $a$ is odd, and $p=2$ or $3.$ \\
  (b) $p$ stays prime if \\
  \indent (i) $\beta=0,$ $a$ is even, and $p \equiv 3~(\textrm{mod}~4)$; \\
  \indent (ii) $\beta=\pm \sqrt{q},$ $a$ is even, and $p \equiv 2 ~(\textrm{mod}~3).$ \\
  (c) $p$ splits in all other cases.
\end{lemma}
\begin{proof}
 For a proof, see \cite[Lemma in $\S$4.1]{15}.
\end{proof}

Then we have the following
\begin{theorem}\label{isogclass ell}
  The isogeny classes of elliptic curves over $k=\F_q$ are in one-to-one correspondence with the rational integers $\beta$ having $|\beta| \leq 2 \sqrt{q}$ and satisfying one of the following conditions: \\
  (1) $\gcd(\beta, p)=1$; \\
  (2) If $a$ is even, then $\beta = \pm 2 \sqrt{q}$; \\
  (3) If $a$ is even and $p \not \equiv 1 ~(\textrm{mod}~3)$, then $\beta = \pm \sqrt{q}$; \\
  (4) If $a$ is odd and $p=2$ or $3$, then $\beta = \pm p^{\frac{a+1}{2}}$; \\
  (5) If either (i) $a$ is odd or (ii) $a$ is even and $p \not \equiv 1~(\textrm{mod}~4),$ then $\beta =0.$ \\
  The first of these are not supersingular; the second are and have all their endomorphisms defined over $k$; the rest are but do not have all their endomorphisms defined over $k.$
\end{theorem}
\begin{proof}
For a proof, see \cite[Theorem 4.1]{15}.
\end{proof}

Using a result from Section \ref{findiv} below, together with Theorem \ref{isogclass ell}, we can derive the following result on the automorphism groups of elliptic curves over finite fields:
\begin{corollary}\label{autoell}
  A finite group $G$ is the automorphism group of an elliptic curve $E$ over a finite field $k=\F_q$ if and only if $G$ is one of the following groups:
  \begin{center}
\begin{tabular}{|c|c|c|c|c|}
\hline
  & $G$ & $p$ \\
\hline
$\sharp 1 $ & $\Z/2\Z$ & $-$\\
\hline
$ \sharp 2$ & $\Z/4\Z$ & $-$ \\
\hline
$\sharp 3 $ &$\Z/6\Z$  & $-$\\
\hline
$\sharp 4 $ & $\textrm{Dic}_{12}$ & $3$\\
\hline
$\sharp 5 $ & $\mathfrak{T}^*$ & $2$ \\
\hline
\end{tabular}
\vskip 4pt
\textnormal{Table 2}
\end{center}
Here, $\textrm{Dic}_{12}$ (resp.\ $\mathfrak{T}^*$) denotes the dicyclic group of order $12$ (resp.\ the binary tetrahedral group).
\end{corollary}

\begin{proof}
 Suppose first that $G$ is the automorphism group of an elliptic curve $E$ over $k=\F_q$. Then we know that $\textrm{End}_k^0(E)$ is either a field with $[\textrm{End}_k^0(E):\Q] \leq 2$ or a definite quaternion algebra over $\Q.$ If $\textrm{End}_k^0(E)$ is a field, then $G$ is cyclic. By \cite[Theorem III.10.1]{12} together with dimension counting, we obtain $\sharp 1, \sharp 2$, and $\sharp 3$. If $\textrm{End}_k^0(E)$ is a definite quaternion algebra over $\Q$, then it is a division algebra, and hence, by Theorem \ref{thm 18} below, together with \cite[Theorem III.10.1]{12} and dimension counting, we see that $G$ can be one of the 5 groups in the above table or $\textrm{Dic}_{24}.$ Then since $G$ can be $\textrm{Dic}_{24}$ only if $\textrm{dim}_{\Q} \textrm{End}_k^0(E) =8$ which is absurd in our situation, we obtain $\sharp 1, \sharp 2, \sharp 3, \sharp 4,$ and $\sharp 5.$ Conversely, we prove that each of the 5 groups occurs as the automorphism group of an elliptic curve $E$ over some finite field $k=\F_q$. \\ 

 (1) Let $k=\F_3.$ By Theorem \ref{isogclass ell}, there is an ordinary elliptic curve over $k$ such that $\textrm{End}_k^0(E)=\Q(\sqrt{-2}).$ Then since $\Z[\sqrt{-2}]$ is a maximal $\Z$-order in $\textrm{End}_k^0(E),$ there is an elliptic curve $E^{\prime}$ over $k$ that is isogenous to $E$ with $\textrm{End}_k(E^{\prime}) = \Z[\sqrt{-2}]$ by Proposition \ref{endposs}. Then since $\textrm{Aut}_k(E^{\prime})=\Z[\sqrt{-2}]^{\times}=\{\pm 1\}$ (with a canonical polarization), we can conclude that $\textrm{Aut}_k(E^{\prime})\cong \Z/2\Z.$  \\

 (2) Let $k=\F_5.$ By Theorem \ref{isogclass ell}, there is an ordinary elliptic curve over $k$ such that $\textrm{End}_k^0(E)=\Q(\sqrt{-1}).$ Then since $\Z[\sqrt{-1}]$ is a maximal $\Z$-order in $\textrm{End}_k^0(E),$ there is an elliptic curve $E^{\prime}$ over $k$ that is isogenous to $E$ with $\textrm{End}_k(E^{\prime}) = \Z[\sqrt{-1}]$ by Proposition \ref{endposs}. Then since $\textrm{Aut}_k(E^{\prime})=\Z[\sqrt{-1}]^{\times}=\{\pm 1, \pm \sqrt{-1}\}$ (with a canonical polarization), we can conclude that $\textrm{Aut}_k(E^{\prime})\cong \Z/4\Z.$ \\

 (3) Let $k=\F_7.$ By Theorem \ref{isogclass ell}, there is an ordinary elliptic curve over $k$ such that $\textrm{End}_k^0(E)=\Q(\sqrt{-3}).$ Then since $\Z \left[\frac{1+\sqrt{-3}}{2}\right]$ is a maximal $\Z$-order in $\textrm{End}_k^0(E),$ there is an elliptic curve $E^{\prime}$ over $k$ that is isogenous to $E$ with $\textrm{End}_k(E^{\prime}) = \Z \left[\frac{1+\sqrt{-3}}{2}\right]$ by Proposition \ref{endposs}. Then since $\textrm{Aut}_k(E^{\prime})=\Z \left[\frac{1+\sqrt{-3}}{2}\right]^{\times}=\{\pm 1, \pm \lambda, \pm \lambda^2 \}$ (with a canonical polarization) where $\lambda = \frac{1+\sqrt{-3}}{2}$, we can conclude that $\textrm{Aut}_k(E^{\prime})\cong \Z/6\Z.$ \\

(4) Let $k=\F_9.$ By Theorem \ref{isogclass ell}, there is a supersingular elliptic curve over $k$ such that $\textrm{End}_k^0(E)=D_{3,\infty}.$ Let $\mathcal{O}$ be a (unique) maximal $\Z$-order in $D_{3,\infty}$ that is generated by $1, i, \frac{1+j}{2}, \frac{i+ij}{2}$. Then there is an elliptic curve $E^{\prime}$ over $k$ that is isogenous to $E$ with $\textrm{End}_k(E^{\prime}) = \mathcal{O}$ by Proposition \ref{endposs}. Then since $\textrm{Aut}_k(E^{\prime})=\mathcal{O}^{\times}$ (with a canonical polarization), we can conclude that $\textrm{Aut}_k(E^{\prime})\cong \textrm{Dic}_{12}.$  \\

(5) Let $k=\F_4.$ By Theorem \ref{isogclass ell}, there is a supersingular elliptic curve over $k$ such that $\textrm{End}_k^0(E)=D_{2,\infty}.$ Let $\mathcal{O}$ be a (unique) maximal $\Z$-order in $D_{2,\infty}$ that is generated by $i, j, ij, \frac{1+i+j+ij}{2}$. Then there is an elliptic curve $E^{\prime}$ over $k$ that is isogenous to $E$ with $\textrm{End}_k(E^{\prime}) = \mathcal{O}$ by Proposition \ref{endposs}. Then since $\textrm{Aut}_k(E^{\prime})=\mathcal{O}^{\times}$ (with a canonical polarization), we can conclude that $\textrm{Aut}_k(E^{\prime})\cong \mathfrak{T}^*.$ \\ 

 This completes the proof.
\end{proof}

For later use, we now determine the possible endomorphism rings of elliptic curves over finite fields.

\begin{theorem}\label{ordoccurell}
  Let $E$ be an elliptic curve over a finite field $k=\F_q$, let $D=\textrm{End}_k^0(E)$, and let $\pi$ be a $q$-Weil number associated to $E$. If $K:=\Q(\pi),$ then the following orders occur as the endomorphism ring of an elliptic curve over $k$ in the isogeny class of $E$: \\
  (1) $E$ is ordinary so that $D=K$: every order containing $\pi$;\\
  (2) $E$ is supersingular with $D=D_{p,\infty}$: every maximal order; \\
  (3) $E$ is supersingular with $D=K$: every order containing $\pi$ whose conductor is relatively prime to $p$.
\end{theorem}
\begin{proof}
 For a proof, see \cite[Theorem 4.2]{15}.
\end{proof}

\section{Quaternionic matrix representations}\label{quat mat rep}

\qquad In this section, we introduce some facts about finite quaternionic matrix groups that will be used later, following a paper of Nebe~\cite{10}. \\

Throughout this section, let $\mathcal{D}$ be a definite quaternion algebra over a totally real number field $K$ and let $V=\mathcal{D}^{1 \times n}$ be a right module over $M_n(\mathcal{D}).$ Endomorphisms of $V$ are given by left multiplication by elements of $\mathcal{D}.$ For computations, it is also convenient to let the endomorphisms act from the right. Then we have $\textrm{End}_{M_n(\mathcal{D})}(V) \cong \mathcal{D}^{\textrm{op}}$, which we identify with $\mathcal{D}$ since $\mathcal{D}$ is a quaternion algebra. \\

Now, we start with the following
\begin{definition}\label{Nebe def}
Let $G$ be a finite group and $\Delta : G \rightarrow GL_n (\mathcal{D})$ be a representation of $G$. \\
(a) Let $L$ be a subring of $K$. The \emph{enveloping $L$-algebra} $\overline{L \Delta(G)}$ is defined as
\begin{equation*}
\overline{L \Delta(G)} = \left\{\sum_{g \in \Delta(G)} l_g \cdot g~|~l_g \in L \right\} \subseteq M_n (\mathcal{D}).
\end{equation*}
(b) $\Delta$ is called \emph{absolutely irreducible} if the enveloping $\Q$-algebra $\overline{\Delta(G)}:=\overline{\Q \Delta(G)}$ of $\Delta(G)$ equals $M_n(\mathcal{D}).$ \\
(c) $\Delta$ is \emph{centrally irreducible} if the enveloping $K$-algebra $\overline{K \Delta(G)}$ equals $M_n(\mathcal{D}).$ \\
(d) $\Delta$ is \emph{irreducible} if the commuting algebra $C_{M_n(\mathcal{D})}(\Delta(G))$ is a division algebra. \\
(e) A subgroup $G \leq GL_n(\mathcal{D})$ is called \emph{irreducible} (resp.\ \emph{centrally irreducible}, \emph{absolutely irreducible}) if the natural representation $\textrm{id} : G \rightarrow GL_n(\mathcal{D})$ is irreducible (resp.\ centrally irreducible, absolutely irreducible).
\end{definition}

\begin{remark}\label{irred max}
(a) The irreducible maximal finite subgroups $G$ of $GL_n(\mathcal{D})$ are absolutely irreducible in their enveloping $\Q$-algebras $\overline{G}$ where $\overline{G} \cong M_m(\mathcal{D}^{\prime})$ for some integer $m \geq 1$ and division algebra $\mathcal{D}^{\prime}$ with $m^2 \cdot \textrm{dim}_{\Q} \mathcal{D}^{\prime} $ dividing $n^2 \cdot \textrm{dim}_{\Q} \mathcal{D}.$ \\
(b) The reducible maximal finite subgroups of $GL_n(\mathcal{D})$ can be built up from the irreducible maximal finite subgroups of $GL_l (\mathcal{D})$ for $l < m.$
\end{remark}

For Lemmas \ref{irred max lem}, \ref{imag lem}, \ref{irre max real mat}, \ref{dihe lem}, and \ref{div alg lem} below, we assume further that $\mathcal{D}=D_{p,\infty}$ for some prime number $p.$ Then in view of Remark \ref{irred max} and the double centralizer theorem, we have
\begin{lemma}\label{irred max lem}
  Let $G$ be an irreducible (but not absolutely irreducible) maximal finite subgroup of $GL_2(\mathcal{D})$. Then $\overline{G}$ is one of the followings:
  \begin{center}
  \begin{tabular}{|c|c|c|c|}
\hline
$$ & $\overline{G}$ & $\textrm{dim}_{\Q} \overline{G}$ & $C_{M_2(\mathcal{D})}(\overline{G})$ \\
\hline
$\sharp 1$ & $M_2(K^{\prime})$ ~$\textrm{where~$K^{\prime}$~is a quadratic field}$ & $8$ & $K^{\prime}$ \\
\hline
$\sharp 2$ & $M_2(\Q)$ & $4$ & $\mathcal{D}$ \\
\hline
$\sharp 3$ & $M_1(\mathcal{D}^{\prime})$~$\textrm{where $\mathcal{D}^{\prime}$ is a division algebra over $\Q$}$ & $2,4, 8$ &   \\
\hline
\end{tabular}
\vskip 4pt
\textnormal{Table 3}
\end{center}
\end{lemma}

We examine each case of Lemma \ref{irred max lem} in the subsequent lemmas:

\begin{example}\label{GL(2,3)}
  Let $G=GL_2(\F_3).$ Then since $13$ is inert in $\Q(\sqrt{-2})$ so that $\Q(\sqrt{-2}) \subset D_{13,\infty},$ we have $G \leq GL_2(D_{13,\infty}).$ Also, it is known in \cite{2} that $G$ is an irreducible maximal finite subgroup of $GL_2(\Q(\sqrt{-2})).$ We claim that $G$ is a maximal finite subgroup of $GL_2(D_{13,\infty})$ (up to isomorphism). Indeed, suppose on the contrary that there is a finite subgroup $H$ of $GL_{2}(D_{13,\infty})$ such that $G$ is (isomorphic to) a proper subgroup of $H$. Since $G$ is an absolutely irreducible maximal finite subgroup of $GL_2(\Q(\sqrt{-2})),$ it follows that $\overline{H}=M_2(D_{13,\infty})$ by dimension counting. This means that $H$ is an absolutely irreducible subgroup of $GL_2(D_{13,\infty}),$ and this contradicts Theorem \ref{prim aimf of GL2} below. Hence, $G$ is a maximal finite subgroup of $GL_2(D_{13,\infty}).$ Also, by definition, $G$ is irreducible since we have $C_{M_2(D_{13,\infty})}(\overline{G})=\Q(\sqrt{-2}),$ which is a division algebra.
\end{example}

By a similar argument as in Example \ref{GL(2,3)}, using the argument of the proof of Theorem \ref{powordelli} below, we have the following:
\begin{lemma}\label{imag lem}
  Let $G$ be an irreducible maximal (up to isomorphism) finite subgroup of $GL_2(\mathcal{D})$ whose enveloping $\Q$-algebra is $M_2(K^{\prime})$ for some imaginary quadratic field $K^{\prime}$. Then $G$ is one of the following groups:
 \begin{center}
    \begin{tabular}{|c|c|}
\hline
$$ & $G$  \\
\hline
$\sharp 1$ & $GL_2(\F_3)$ \\
\hline
$\sharp 2$ & $\Z/12 \rtimes \Z/2\Z \cong \Z/4\Z \times \textrm{Sym}_3$  \\
\hline
$\sharp 3$ & $(Q_8 \rtimes \Z/3\Z) \times \Z/3\Z \cong \mathfrak{T}^* \times \Z/3\Z$  \\
\hline
$\sharp 4$ & $ (\Z/6\Z \times \Z/6\Z) \rtimes \Z/2\Z \cong \textrm{Dic}_{12} \rtimes \Z/6\Z$    \\
\hline
$\sharp 5$ & $(Q_8 \rtimes \Z/3\Z) \cdot \Z/4\Z \cong \mathfrak{T}^* \rtimes \Z/4\Z$ \\
\hline
\end{tabular}
\vskip 4pt
\textnormal{Table 4}
\end{center}
\end{lemma}

If $G \leq GL_2(K^{\prime})$ for some totally real quadratic field $K^{\prime}$, then $G$ is either cyclic or a dihedral group by a well-known theorem of Leonardo da Vinci. By considering the possible orders of $G$ via dimension counting, we also have $|G|\leq 24.$\footnote{In fact, if $G$ is cyclic, then $|G| \leq 12.$} This observation gives the following result:
\begin{lemma}\label{irre max real mat}
  There is no irreducible maximal (up to isomorphism) finite subgroup $G$ of $GL_2(\mathcal{D})$ whose enveloping $\Q$-algebra is $M_2(K^{\prime})$ for some totally real quadratic field $K^{\prime}$.
\end{lemma}
\begin{proof}
In the sequel, let $D_{n}$ denote the dihedral group of order $2n.$ By the assumption on the enveloping $\Q$-algebra, all the possible such groups are $G=D_5, D_8, D_{10},$ or $D_{12}.$ We prove that those groups cannot occur by considering them one by one. \\

(1) Let $G=D_{5}$. Suppose that $G$ is an irreducible maximal finite subgroup of $GL_2(D_{p,\infty})$ for some prime number $p.$ Then by definition, $\overline{G}=M_2(\Q(\sqrt{5})) \subset M_2(D_{p,\infty}).$ By the double centralizer theorem, we have $C_{M_2({D_{p,\infty}})}(\overline{G})=\Q(\sqrt{5})$ and $M_2(\Q(\sqrt{5})) \sim M_2(D_{p,\infty} \otimes_{\Q} \Q(\sqrt{5})).$ It follows that $D_{p,\infty}$ splits over $\Q(\sqrt{5})$, which is equivalent to $\Q(\sqrt{5}) \hookrightarrow D_{p,\infty}.$ Since the real place of $\Q$ splits in $\Q(\sqrt{5}),$ this is a contradiction. Hence, $G$ cannot occur. \\

(2) Let $G=D_8$. Suppose that $G$ is an irreducible maximal finite subgroup of $GL_2(D_{p,\infty})$ for some prime number $p.$ Then by definition, $\overline{G}=M_2(\Q(\sqrt{2})) \subset M_2(D_{p,\infty}),$ and a similar argument as in the proof of (1) can be used to prove that $G$ cannot occur as such a group. \\ 

(3) Let $G=D_{10}$. Then a similar argument as in the proof of (1) can be used to show that $G$ cannot occur as such a group. \\

(4) Let $G=D_{12}$. Suppose that $G$ is an irreducible maximal finite subgroup of $GL_2(D_{p,\infty})$ for some prime number $p.$ Then by definition, $\overline{G}=M_2(\Q(\sqrt{3})) \subset M_2(D_{p,\infty}),$ and a similar argument as in the proof of (1) can be used to prove that $G$ cannot occur as such a group. \\ 

This completes the proof.
\end{proof}

We also have

\begin{lemma}\label{dihe lem}
  Let $G$ be an irreducible maximal (up to isomorphism) finite subgroup of $GL_2(\mathcal{D})$ whose enveloping $\Q$-algebra is $M_2(\Q)$. Then $G$ is one of the following groups:
  \begin{center}
    \begin{tabular}{|c|c|}
\hline
$$ & $G$  \\
\hline
$\sharp 1$ & $D_4$ \\
\hline
$\sharp 2$ & $D_6$  \\
\hline
\end{tabular}
\vskip 4pt
\textnormal{Table 5}
\end{center}
\end{lemma}
\begin{proof}
It is well-known that $D_4$ and $D_6$ are maximal finite subgroups of $GL_2(\Q).$ For later use, we note that $D_4$ and $D_6$ are irreducible maximal (up to isomorphism) finite subgroups of $GL_2(D_{241,\infty}).$
\end{proof}

Finally, if $\overline{G}=M_1(\mathcal{D}^{\prime})$ for some division algebra $\mathcal{D}^{\prime}$ over $\Q$ with $\textrm{dim}_{\Q} \mathcal{D}^{\prime} \in \{2,4,8\},$ then, in view of Theorem \ref{thm 18} below, we have

\begin{lemma}\label{div alg lem}
  Let $G$ be an irreducible maximal (up to isomorphism) finite subgroup of $GL_2(\mathcal{D})$ whose enveloping $\Q$-algebra is $M_1(\mathcal{D}^{\prime})$ for some division algebra $\mathcal{D}^{\prime}$ over $\Q$ with $\textrm{dim}_{\Q} \mathcal{D}^{\prime} \in \{2,4,8\}$. Then $G$ is one of the following groups:
  \begin{center}
    \begin{tabular}{|c|c|}
\hline
$$ & $G$  \\
\hline
$\sharp 1$ & $\mathfrak{I}^* \cong SL_2 (\F_5)$ \\
\hline
$\sharp 2$ & $\textrm{Dic}_{24}$  \\
\hline
$\sharp 3$ & $\mathfrak{O}^*$  \\
\hline
$\sharp 4$ & $ \mathfrak{T}^* \cong SL_2 (\F_3)$    \\
\hline
$\sharp 5$ & $\textrm{Dic}_{12}$ \\
\hline
\end{tabular}
\vskip 4pt
\textnormal{Table 6}
\end{center}
\end{lemma}
\begin{proof}
Note that since $\textrm{dim}_{\Q} M_2(D_{p,\infty}) = 16,$ if $\mathcal{D}^{\prime}$ is a field, then $\textrm{dim}_{\Q} \mathcal{D}^{\prime} \leq 4.$ Also, we know that every finite subgroup of the multiplicative subgroup of a field is cyclic. If $\mathcal{D}^{\prime}$ is not a field, then, in light of Theorem \ref{thm 18} below, all the possible maximal finite subgroups are listed in the above table. Now, we prove that those groups occur by considering them one by one. \\

(1) Since $7$ is inert in $\Q(\sqrt{5}),$ we can see that $\left( \begin{array}{cc}-1,-1 \\ \Q(\sqrt{5}) \end{array} \right) \subseteq M_2(D_{7,\infty}).$ We also have that $\mathfrak{I}^*$ is an absolutely irreducible maximal finite subgroup of $\left( \begin{array}{cc}-1,-1 \\ \Q(\sqrt{5}) \end{array} \right)$ (see Theorem \ref{aimf of GL1} below). Hence, $\mathfrak{I}^*$ is an irreducible maximal finite subgroup of $GL_2(D_{7,\infty}).$ \\

(2) Since $7$ is inert in $\Q(\sqrt{3}),$ we can see that $\left( \begin{array}{cc}-1,-1 \\ \Q(\sqrt{3}) \end{array} \right) \subseteq M_2(D_{7,\infty}).$ We also have that $\textrm{Dic}_{24}$ is an absolutely irreducible maximal finite subgroup of $\left( \begin{array}{cc}-1,-1 \\ \Q(\sqrt{3}) \end{array} \right)$ (see Theorem \ref{aimf of GL1} below). Hence, $\textrm{Dic}_{24}$ is an irreducible maximal finite subgroup of $GL_2(D_{7,\infty}).$ \\

(3) Since $11$ is inert in $\Q(\sqrt{2}),$ we can see that $\left( \begin{array}{cc}-1,-1 \\ \Q(\sqrt{2}) \end{array} \right) \subseteq M_2(D_{11,\infty}).$ We also have that $\mathfrak{O}^*$ is an absolutely irreducible maximal finite subgroup of $\left( \begin{array}{cc}-1,-1 \\ \Q(\sqrt{2}) \end{array} \right)$ (see Theorem \ref{aimf of GL1} below). Hence, $\mathfrak{O}^*$ is an irreducible maximal finite subgroup of $GL_2(D_{11,\infty}).$ \\

(4) Note that $D_{2,\infty} \subseteq M_2(D_{241,\infty})$. We also have that $\mathfrak{T}^*$ is an absolutely irreducible maximal finite subgroup of $D_{2,\infty}$ (see Theorem \ref{aimf of GL1} below). Hence, $\mathfrak{T}^*$ is an irreducible maximal finite subgroup of $GL_2(D_{241,\infty}).$ \\

(5) Note that $D_{3,\infty} \subseteq M_2(D_{241,\infty})$. We also have that $\textrm{Dic}_{12}$ is an absolutely irreducible maximal finite subgroup of $D_{3,\infty}$ (see Theorem \ref{aimf of GL1} below). Hence, $\textrm{Dic}_{12}$ is an irreducible maximal finite subgroup of $GL_2(D_{241,\infty}).$ \\

This completes the proof.
\end{proof}

Now, as in the case of $\mathcal{D}=\Q,$ the notion of primitivity gives an important reduction in the determination of the maximal finite subgroups of $GL_n(\mathcal{D}).$

\begin{definition}\label{prim def}
  Let $G $ be an irreducible finite subgroup of $GL_n(\mathcal{D})$. Consider $V:=\mathcal{D}^{1 \times n}$ as a $\mathcal{D}$-$G$-bimodule. Then $G$ is called \emph{imprimitive} if there is a decomposition $V=V_1 \oplus \cdots \oplus V_s$ of $V$ as a nontrivial direct sum of $\mathcal{D}$-left modules such that $G$ permutes the $V_i$ (i.e.\ for all $g \in G$ and for all $1 \leq i \leq s,$ there is a $j \in \{1,\cdots,s\}$ such that $V_i g \subseteq V_j$). If $G$ is not imprimitive, then $G$ is called \emph{primitive}.
\end{definition}

The following remark gives a way to produce the imprimitive absolutely irreducible maximal finite subgroups of $GL_n(\mathcal{D})$ from the primitive absolutely irreducible maximal finite subgroups of smaller degree.

\begin{remark}\label{impri and prim}
  Let $G $ be an imprimitive absolutely irreducible maximal finite subgroup of $GL_n(\mathcal{D})$. Since $G$ is maximal and finite, it follows that $G$ is a wreath product of primitive absolutely irreducible maximal finite subgroups of $GL_d(\mathcal{D})$ with the symmetric group $\textrm{Sym}_{\frac{n}{d}}$ of degree $\frac{n}{d}$ for divisors $d$ of $n.$ In particular, every imprimitive absolutely irreducible maximal finite subgroup of $GL_2(\mathcal{D})$ is a wreath product of primitive absolutely irreducible maximal finite subgroups of $GL_1(\mathcal{D})$ with the symmetric group $\textrm{Sym}_2.$
\end{remark}

To introduce our next result, let $^* : \mathcal{D} \rightarrow \mathcal{D}$ be the canonical involution of $\mathcal{D}$ such that $x x^* \in K$ for all $x \in \mathcal{D}$ and such that $^*$ induces the identity map on the center $K.$ We extend the involution $^*$ to $M_n (\mathcal{D})$ by applying $^*$ to the entries of the matrices. Then the map $g \mapsto (g^*)^t$ is an involution on $M_n(\mathcal{D})$ where $g^t$ denotes the transpose of the matrix $g$. In this situation, we can describe the maximal finite subgroups of $GL_n (\mathcal{D})$ as full automorphism groups of totally positive definite Hermitian lattices as in the following lemma:

\begin{definition and lemma}\label{main lem hard}
Let $G$ be a finite subgroup of $GL_n (\mathcal{D})$, let $V=\mathcal{D}^{1 \times n}$ be the natural $G$-right module, and let $\mathfrak{m}$ be an order in $\mathcal{D}=\textrm{End}_{M_n(\mathcal{D})}(V).$ Then we have: \\
(a) An \emph{$\mathfrak{m}$-lattice} $L \leq V$ is a finitely generated projective $\mathfrak{m}$-left module with $\mathbb{Q}L=V.$ \\
(b) The set
\begin{equation*}
Z_{\mathfrak{m}}(G):=\{L \leq V~|~L~\textrm{is an}~\mathfrak{m}\textrm{-lattice and}~LG \subseteq L\}
\end{equation*}
of $G$-invariant $\mathfrak{m}$-lattices in $V$ is nonempty. \\
(c) The $K$-vector space
\begin{equation*}
\mathcal{F}(G):=\{F \in M_n(\mathcal{D})~|~F^t =F^*~\textrm{and}~gF(g^*)^t = F~\textrm{for all}~g \in G\}
\end{equation*}
of $G$-invariant Hermitian forms contains a totally positive definite form, i.e.\ the set
\begin{equation*}
\mathcal{F}_{>0}(G):=\{F \in \mathcal{F}(G)~|~\epsilon(F)~\textrm{is positive definite for all embeddings}~\epsilon : K \hookrightarrow \R\}
\end{equation*}
is nonempty. \\
(d) Let $L$ be an $\mathfrak{m}$-lattice in $V$ and $F \in M_n(\mathcal{D})$ a totally positive definite Hermitian form. The automorphism group
\begin{equation*}
\textrm{Aut}(L,F):=\{g \in GL_n(\mathcal{D})~|~Lg \subseteq L~\textrm{and}~g F(g^*)^t =F \}
\end{equation*}
of $L$ with respect to $F$ is a finite group. \\
(e) The absolutely irreducible maximal finite supergroups of $G$ are of the form $\textrm{Aut}(L,F)$ for some $L \in Z_{\mathfrak{m}}(G)$ and $F \in \mathcal{F}_{>0}(G).$
\end{definition and lemma}
\begin{proof}
For a proof, see \cite[Definition and Lemma 2.6]{10}.
\end{proof}
In light of Definition and Lemma \ref{main lem hard}-(e), we may compute all absolutely irreducible maximal finite supergroups of a finite subgroup $G$ of $GL_n(\mathcal{D})$ as automorphism groups of $G$-invariant lattices. \\

Now, we introduce two results about absolutely irreducible maximal finite subgroups of $GL_n(\mathcal{D})$ with $n \leq 2.$ \\

Suppose first that $n=1.$ By using a well-known classification of finite subgroups of $PGL_2(\C)$, we have the following
\begin{theorem}\label{aimf of GL1}
Let $G $ be an absolutely irreducible maximal finite subgroup of $GL_1(\mathcal{D})$. Then we have: \\
(a) $K$ is the maximal totally real subfield of a cyclotomic field. \\
(b) If $[K : \mathbb{Q}] \leq 2,$ then $G$ is one of the following groups (up to isomorphism):
\begin{center}
  \begin{tabular}{|c|c|c|}
\hline
$$ & $G$ & $\mathcal{D}$ \\
\hline
$\sharp 1$ & $\mathfrak{T}^* \cong SL_2(\F_3)$ & $D_{2,\infty}$ \\
\hline
$\sharp 2$ & $\widetilde{\textrm{Sym}_3} =\textrm{Dic}_{12}$ & $D_{3,\infty}$  \\
\hline
$\sharp 3$ & $\widetilde{\textrm{Sym}_4} = \mathfrak{O}^*$ & $D_{\sqrt{2},\infty}$  \\
\hline
$\sharp 4$ & $Q_{24}=\textrm{Dic}_{24} $ & $D_{\sqrt{3},\infty}$    \\
\hline
$\sharp 5$ & $\mathfrak{I}^* \cong SL_2(\F_5) $ & $D_{\sqrt{5},\infty}$ \\
\hline
\end{tabular}
\vskip 4pt
\textnormal{Table 7}
\end{center}
where $D_{\alpha,\infty}$ denotes the definite quaternion algebra over $K$ such that $K=\mathbb{Q}(\alpha)$ and $D_{\alpha, \infty}$ is ramified at the primes $\alpha$ and $\infty.$
\end{theorem}
\begin{proof}
For a proof, see \cite[Theorem 6.1]{10}.
\end{proof}

Next, we suppose that $n=2.$ For our own purpose, we only consider the case when $K=\mathbb{Q}.$ Then the following important result is obtained:

\begin{theorem}\label{prim aimf of GL2}
Suppose that $K=\mathbb{Q}$ and let $G $ be a primitive absolutely irreducible maximal finite subgroup of $ GL_2(\mathcal{D})$. Then $\mathcal{D}$ is one of $D_{2,\infty}, D_{3,\infty},$ or $D_{5,\infty}$ and $G$ is conjugate to one of the following groups:
\begin{center}
 \begin{tabular}{|c|c|c|}
\hline
$$ & $G$ & $\mathcal{D}$ \\
\hline
$\sharp 1$ & $2_{-}^{1+4}.\textrm{Alt}_5$ & $D_{2,\infty}$ \\
\hline
$\sharp 2$ & $SL_2(\F_3)\times \textrm{Sym}_3$ & $D_{2,\infty}$  \\
\hline
$\sharp 3$ & $SL_2(\F_9)$ & $D_{3,\infty}$  \\
\hline
$\sharp 4$ & $ \Z/3\Z : (SL_2(\F_3).2)$ & $D_{3,\infty}$    \\
\hline
$\sharp 5$ & $SL_2(\F_5).2 $ & $D_{5,\infty}$ \\
\hline
$\sharp 6$ & $SL_2(\F_5):2 $ & $D_{5,\infty}$ \\
\hline
\end{tabular}
\vskip 4pt
\textnormal{Table 8}
\end{center}
\end{theorem}

\begin{proof}
For a proof, see \cite[Theorem 12.1]{10}.
\end{proof}

In view of Remark \ref{impri and prim} and Theorem \ref{aimf of GL1}, we also have the following
\begin{theorem}\label{imprim aimf GL2}
Suppose that $K=\mathbb{Q}$ and let $G $ be an imprimitive absolutely irreducible maximal finite subgroup of $ GL_2(\mathcal{D})$. Then $\mathcal{D}$ is one of $D_{2,\infty}$ or $D_{3,\infty}$ and $G$ is one of the following groups:
\begin{center}
 \begin{tabular}{|c|c|c|}
\hline
$$ & $G$ & $\mathcal{D}$ \\
\hline
$\sharp 1$ & $(SL_2(\F_3))^2 \rtimes \textrm{Sym}_2 $ & $D_{2,\infty}$ \\
\hline
$\sharp 2$ & $(\textrm{Dic}_{12})^2 \rtimes \textrm{Sym}_2$ & $D_{3,\infty}$  \\
\hline
\end{tabular}
\vskip 4pt
\textnormal{Table 9}
\end{center}
\end{theorem}

\section{Finite Subgroups of Division Algebras}\label{findiv} 
\qquad Recall that if $X$ is a simple abelian surface over a finite field $k$, then $\textrm{End}_k^0(X)$ is a division algebra over $\Q$ with either $\textrm{dim}_{\Q}\textrm{End}_k^0(X)=4$ or $\textrm{dim}_{\Q}\textrm{End}_k^0(X)=8$. Hence, in this section, we give a classification of all possible finite groups that can be embedded in the multiplicative group of a division algebra over $\Q$ with certain properties that are related to our situation. Our main reference is a paper of Amitsur~\cite{1}. We start with the following

\begin{definition}\label{def 9}
  Let $m, r$ be two relatively prime positive integers, and we put $s:=\gcd(r-1, m)$ and $t:=\frac{m}{s}.$ Also, let $n$ be the smallest integer such that $r^n \equiv 1~(\textrm{mod}~m).$ We denote by $G_{m,r}$ the group generated by two elements $a,b$ satisfying the relations
  \begin{equation*}
    a^m=1,~b^n=a^t,~ bab^{-1}=a^r.
  \end{equation*}
  This group is called the \emph{dicyclic group} of order $mn.$ Alternatively, we often write $\textrm{Dic}_{mn}$ for $G_{m,r}$ in this case. As a convention, if $r=1,$ then we put $n=s=1,$ and hence, $G_{m,1}$ is a cyclic group of order $m.$
\end{definition}

Given $m,r,s,t,n,$ as above, we will consider the following two conditions in the sequel: \\
  (C1) $\gcd(n,t)=\gcd(s,t)=1.$ \\
  (C2) $n=2n^{\prime}, m=2^{\alpha} m^{\prime}, s=2 s^{\prime}$ where $\alpha \geq 2,$ and $n^{\prime}, m^{\prime}, s^{\prime}$ are all odd integers. Moreover, $\gcd(n,t)=\gcd(s,t)=2$ and $r \equiv -1~(\textrm{mod}~2^{\alpha}).$ \\

Now, let $p$ be a prime number that divides $m.$ We define: \\
(i) $\alpha_p$ is the largest integer such that $p^{\alpha_p}~|~m.$ \\
(ii) $n_p$ is the smallest integer satisfying $r^{n_p} \equiv 1~(\textrm{mod}~mp^{-\alpha_p}).$ \\
(iii) $\delta_p$ is the smallest integer satisfying $p^{\delta_p} \equiv 1~(\textrm{mod}~mp^{-\alpha_p}).$ \\

Then we have the following

\begin{lemma}\label{lem 10}
  Let $q$ be a prime number dividing $n.$ Then there exists at most one prime number $p~|~m$ for which $q \nmid n_p.$ If such a $p$ exists and if \\
  (1) $p \ne 2,$ then $p \nmid s$ and $q~|~p-1$; \\
  (2) $p=2,$ then $q=p=2$ and (C2) holds. \\
  Moreover, if such a $p$ exists for $q=2$, then (C1) implies that (1) holds, and (C2) implies that $p=2$ (i.e.\ (2) is valid).
\end{lemma}
\begin{proof}
  For a proof, see \cite[Lemma 6]{1}.
\end{proof}

Our next theorem provides us with a useful criterion for a dicyclic group to be embedded in a division ring:
\begin{theorem}\label{thm 11}
  A group $G_{m,r}$ can be embedded in a division ring if and only if either (C1) or (C2) holds, and one of the following conditions holds: \\
  (1) $n=s=2$ and $r \equiv -1~(\textrm{mod}~m)$. \\
  (2) For every prime number $q~|~n,$ there exists a prime number $p~|~m$ such that $q \nmid n_p$ and that either \\
  (a) $p \ne 2$, and $\gcd(q,(p^{\delta_p}-1)/s)=1$, or \\
  (b) $p=q=2,$ (C2) holds, and $m/4 \equiv \delta_p \equiv 1~(\textrm{mod}~2).$
\end{theorem}
\begin{proof}
  For a proof, see \cite[Theorem 3, Theorem 4, and Lemma 10]{1}.
\end{proof}

Now, let $G$ be a finite group. One of our main tools in this section is the following result:
\begin{theorem}\label{thm 12}
  $G$ can be embedded in a division ring if and only if $G$ is of one of the following types: \\
  (1) Cyclic groups. \\
  (2) $G_{m,r}$ where the integers $m,r,$ etc, satisfy Theorem \ref{thm 11} (which is not cyclic). \\
  (3) $\mathfrak{T}^* \times G_{m,r}$ where $\mathfrak{T}^*$ is the binary tetrahedral group of order $24$, and $G_{m,r}$ is either cyclic of order $m$ with $\gcd(m,6)=1$, or of type (2) with $\gcd(|G_{m,r}|, 6)=1.$ In both cases, for all primes $p~|~m,$ the smallest integer $\gamma_p$ satisfying $2^{\gamma_p} \equiv 1~(\textrm{mod}~p)$ is odd. \\
  (4) $\mathfrak{O}^*$, the binary octahedral group of order $48$. \\
  (5) $\mathfrak{I}^*,$ the binary icosahedral group of order $120.$
\end{theorem}
\begin{proof}
  For a proof, see \cite[Theorem 7]{1}.
\end{proof}

Two other useful facts are listed below:
\begin{lemma}\label{lem 13}
  If a finite group $G$ can be embedded in a division ring and if $|G|$ is square-free, then $G$ is cyclic.
\end{lemma}
\begin{proof}
  For a proof, see \cite[Corollary 5]{1}.
\end{proof}

\begin{theorem}\label{thm 14}
  Let $G=G_{m,n,r}=\langle a,b~|~a^m=b^n =1, bab^{-1}=a^r \rangle$ where $n>1.$ If the order of $\overline{r}$ in the unit group $U(\Z/m\Z)$ is exactly $n,$ then $G$ cannot be embedded in a division ring.
\end{theorem}
\begin{proof}
  For a proof, see \cite[Theorem 3.1]{6}.
\end{proof}

Having stated all necessary results, we can proceed to achieve our goal of this section. First, we give two important lemmas:

\begin{lemma}\label{lem 15}
  Suppose that $n=2,$ $m$ is an integer with $2 \leq m \leq 12,$ and $\gcd(n,t)=\gcd(s,t)=1.$ Then the group $G:=G_{m,r}$ can be embedded in a division ring if and only if $G$ is one of the following groups: \\
  (1) A cyclic group $\Z/4\Z$\footnote{If we want $G_{m,r}$ to be non-cyclic as in Theorem \ref{thm 12}, then we may exclude this case.}; \\
  (2) $\textrm{Dic}_{12}$, a dicyclic group of order $12$; \\
  (3) $\textrm{Dic}_{20}$, a dicyclic group of order $20$.
\end{lemma}
\begin{proof}
  Note that $t$ is an odd integer since $n=2$ and $\gcd(n,t)=1.$ Hence, we have the following three cases to consider: \\

  (1) If $m \in \{3,5,7,11\},$ then $|G|=2m$ is square-free. Now, since $n=2$ and $m$ is an odd prime number, it is easy to see that $r \equiv -1~(\textrm{mod}~m)$ (by the definition of $n$), and then by looking at the presentation of $G,$ we can see that $G$ is not cyclic. By Lemma \ref{lem 13}, $G$ cannot be embedded in a division ring. \\

  (2) Suppose that $m=9.$ Then since $\gcd(s,t)=1,$ we cannot have $s=t=3.$ Also, since $n=2$, we cannot have $s=9, t=1$ by the definition of $n.$ Hence, the only possible case is that $s=1, t=9.$ Since $n=2$ and $\gcd(r+1,r-1) \leq 2,$ we have $r \equiv -1~(\textrm{mod}~9)$ so that the order of $\overline{r}$ in the unit group $U(\Z/9\Z)$ is exactly $2.$ By Theorem \ref{thm 14}, $G$ cannot be embedded in a division ring. \\

  (3) If $m \in \{2,4,6,8,10,12\},$ then we have the following six subcases to consider: \\
  (i) If $m=2,$ then since $\gcd(n,t)=1,$ we get $s=2, t=1.$ In particular, we have $n=s=2$ and $r \equiv -1~(\textrm{mod}~2).$ By Theorem \ref{thm 11}, $G$ can be $G_{2,r}=\Z/4\Z.$ \\
  (ii) If $m=4,$ then since $\gcd(n,t)=1,$ we get $s=4, t=1.$ By the definition of $s$, we have $m=4~|~r-1$, and hence, this case cannot occur because of our assumption that $n=2.$ \\
  (iii) If $m=6,$ then since $\gcd(n,t)=1,$ either $s=6, t=1$ or $s=2, t=3.$ If $s=6, t=1,$ then by a similar argument as in (ii), this case cannot occur. Now, if $s=2, t=3,$ then since $n=2,$ we have $r \equiv -1~(\textrm{mod}~6).$ By Theorem \ref{thm 11}, $G$ can be $G_{6,r}=\textrm{Dic}_{12}.$ \\
  (iv) If $m=8,$ then since $\gcd(n,t)=1,$ we get $s=8, t=1.$ This case cannot occur by a similar argument as in (ii). \\
  (v) Similarly, if $m=10,$ then we only need to consider the case when $s=2, t=5.$ Since $n=2,$ we have $r \equiv -1~(\textrm{mod}~10).$ By Theorem \ref{thm 11}, $G$ can be $G_{10,r}=\textrm{Dic}_{20}.$ \\
  (vi) If $m=12,$ then, as before, we only need to consider the case when $s=4, t=3.$ If there exists a prime number $p~|~12$ such that $q=2~\nmid~n_p,$ then either $p=2$ or $p=3.$ Since (C2) does not hold, $p \ne 2.$ If $p=3,$ then $\alpha_3 = 1, \delta_3 = 2, $ and hence, we have $\gcd(2, (3^2 -1)/4)=2.$ By Theorem \ref{thm 11}, this case cannot occur. \\

  This completes the proof.
\end{proof}

\begin{lemma}\label{lem 16}
  Suppose that $n=2, m=2^{\alpha} \cdot m^{\prime} \leq 12, s = 2 s^{\prime}, \gcd(s,t)=2,$ and $r \equiv -1~(\textrm{mod}~2^{\alpha})$ where $\alpha \geq 2$ is an integer and $m^{\prime}, s^{\prime}$ are odd integers. Then the group $G:=G_{m,r}$ can be embedded in a division ring if and only if $G$ is one of the following groups:\\
  (1) $Q_8,$ a quaternion group; \\
  (2) $\textrm{Dic}_{16},$ a dicyclic group of order $16$; \\
  (3) $\textrm{Dic}_{24},$ a dicyclic group of order $24.$
\end{lemma}
\begin{proof}
  We have the following three cases to consider: \\

  (1) If $\alpha =2, m^{\prime}=1,$ then $m=4$ so that $s=t=2.$ In particular, $n=s=2$ and $r \equiv -1~(\textrm{mod}~4),$ and hence, $G$ can be $G_{4,r}=Q_8$ by Theorem \ref{thm 11}. \\

  (2) If $\alpha =2, m^{\prime}=3,$ then $m=12$ so that either $s=2, t=6$ or $s=6, t=2.$ Also, $r \equiv -1~(\textrm{mod}~4).$ We have the following three subcases to consider: \\
  (i) If $s=2, t=6,$ and $r \equiv -1~(\textrm{mod}~12),$ then $G$ can be $G_{12,r}=\textrm{Dic}_{24}$ by Theorem \ref{thm 11}.\\
  (ii) Suppose that $s=2, t=6,$ and $r \equiv -5~(\textrm{mod}~12).$ If there exists a prime number $p~|~12$ such that $q=2 \nmid n_p,$ then either $p=2$ or $p=3.$ If $p=2,$ then $\alpha_2 =\delta_2 = 2$, and this case cannot occur by Theorem \ref{thm 11}. If $p=3,$ then $\alpha_3 =1, \delta_3 =2$ so that $\gcd(2,(3^2-1)/2)=2,$ and again, by Theorem \ref{thm 11}, this case cannot occur. \\
  (iii) Suppose that $s=6, t=2$. If there exists a prime number $p~|~12$ such that $q=2 \nmid n_p,$ then either $p=2$ or $p=3.$ By Lemma \ref{lem 10}, we have $p \ne 3.$ If $p=2,$ then we have $\alpha_2 = \delta_2 = 2,$ and hence, this case cannot occur by Theorem \ref{thm 11}. \\

  (3) If $\alpha=3, m^{\prime}=1,$ then $m=8$ so that $s=2, t=4.$ In particular, $n=s=2$ and $r \equiv -1~(\textrm{mod}~8),$ and hence, $G$ can be $G_{8,r}=\textrm{Dic}_{16}$ by Theorem \ref{thm 11}. \\

  This completes the proof.
\end{proof}

Now, let $D$ be a division algebra of degree 2 over its center $K$ (i.e.\ $[D:K]=4$) where $K$ is an algebraic number field. Suppose further that $\textrm{dim}_{\Q}D \leq 8$ so that the order of every element of finite order of $D^{\times}$ is at most $12.$ If the group $G:=G_{m,r}$ is contained in $D^{\times},$ then $n~|~2$ (see \cite[\S7]{1}), and hence, we have either $n=1$ or $n=2.$ Also, since $G$ contains an element of order $m,$ we have $m \leq 12.$ The last preliminary result that we need is the following
\begin{theorem}\label{thm 17}
  Let $D$ be as in the above. \\
 (a) If $D$ contains an $\mathfrak{O}^*,$ then $\sqrt{2}\in K.$ \\
 (b) If $D$ contains an $\mathfrak{I}^*,$ then $\sqrt{5}\in K.$
\end{theorem}
\begin{proof}
  For a proof, see \cite[Theorem 10]{1}.
\end{proof}

Summarizing, we have the following
\begin{theorem}\label{thm 18}
  Let $D$ be as in the above. A finite group $G$ (of even order\footnote{This assumption can be made based on the goal of this paper.}) can be embedded in $D^{\times}$ if and only if $G$ is one of the following groups: \\
  (1) $\Z/2\Z, \Z/4\Z, \Z/6\Z, \Z/8\Z, \Z/10\Z, \Z/12\Z$; \\
  (2) $Q_8, \textrm{Dic}_{12}, \textrm{Dic}_{16}, \textrm{Dic}_{20}, \textrm{Dic}_{24}$; \\
  (3) $\mathfrak{T}^*$; \\
  (4) $\mathfrak{O}^*$ only if $\sqrt{2} \in K$; \\
  (5) $\mathfrak{I}^*$ only if $\sqrt{5} \in K$.
\end{theorem}
\begin{proof}
  We refer the list of possible such groups to Theorem \ref{thm 12}. Suppose that $G$ is cyclic. Then we can write $G=\langle g\rangle$ for some element $g$ of order $d$. Then according to the argument given before Theorem \ref{thm 17}, we have $d \in \{2,4,6,8,10,12\}$. Hence, we obtain (1). If $G=G_{m,r}$ with $n=2$ and $m \leq 12,$ then (2) follows from Lemmas \ref{lem 15} and \ref{lem 16}. (If $n=1,$ then $s=m$ so that $t=1.$ Now, the presentation of the group tells us that the group is cyclic of order $m$ in this case.) Now, if $G=\mathfrak{T}^* \times G_{m,r}$ is a general $T$-group, then the only possible such groups are $\mathfrak{T}^*$ and $\mathfrak{T}^* \times \Z/5\Z.$ (If $n=2$ so that $G_{m,r}$ is not cyclic, then $|G_{m,r}|=2m,$ and hence, we have $\gcd(|G_{m,r}|,6) \ne 1.$) In the latter case, we have $\gamma_5 =4$ which contradicts the condition of Theorem \ref{thm 12}. Also, clearly, $\mathfrak{T}^*$ satisfies the condition of Theorem \ref{thm 12}, and hence, we get (3). Finally, the other possibilities for $G$ are $\mathfrak{O}^*$ and $\mathfrak{I}^*.$ If $G=\mathfrak{O}^*,$ then $|G|=48$, and in this case, $\sqrt{2} \in K$ by Theorem \ref{thm 17}. If $G=\mathfrak{I}^*,$ then $|G|=120,$ and in this case, $\sqrt{5} \in K$ by Theorem \ref{thm 17}. \\

  This completes the proof.
\end{proof}

Regarding our goal of this paper, the above theorem has a nice consequence:
\begin{corollary}\label{cor 19}
  Let $X$ be a simple abelian surface over a finite field $k.$ Let $G$ be a finite subgroup of the multiplicative subgroup of $\textrm{End}_k^0(X).$ Then $G$ is one of the following groups: \\
  (1) $\Z/2\Z, \Z/4\Z, \Z/6\Z, \Z/8\Z, \Z/10\Z, \Z/12\Z$; \\
  (2) $Q_8, \textrm{Dic}_{12}, \textrm{Dic}_{16}, \textrm{Dic}_{20}, \textrm{Dic}_{24}$; \\
  (3) $\mathfrak{T}^*$; \\
  (4) $\mathfrak{O}^*$ only if $\textrm{dim}_{\Q}\textrm{End}_k^0(X)=8$ and $\sqrt{2} \in \Q(\pi_X)$; \\
  (5) $\mathfrak{I}^*$ only if $\textrm{dim}_{\Q}\textrm{End}_k^0(X)=8$ and $\sqrt{5} \in \Q(\pi_X)$.
\end{corollary}
\begin{proof}
  Since $X$ is simple, we have that $\textrm{End}_k^0(X)$ is a division algebra over $\Q$ with either $\textrm{dim}_{\Q}\textrm{End}_k^0(X)=4$ or $\textrm{dim}_{\Q}\textrm{End}_k^0(X)=8$. \\

If $\textrm{dim}_{\Q}\textrm{End}_k^0(X)=4$, then $\textrm{End}_k^0(X)=\Q(\pi_X)$ is a field by Corollary \ref{cor TateEnd0}, and hence, $G$ is cyclic. If $G=\langle g \rangle $ for some element $g$ of order $d,$ then $d \leq 12,$ and hence, this gives the possibility of (1). If $\textrm{dim}_{\Q}\textrm{End}_k^0(X)=8$, then $\textrm{End}_k^0(X)$ is a division algebra of degree 2 over the center $\Q(\pi_X)$ (which is an algebraic number field), and hence, $G$ is of one of the given types (1), (2), (3), (4), and (5) by Theorem \ref{thm 18}. \\

  This completes the proof.
\end{proof}

\section{Main Result}\label{main}
\qquad In this section, we will give a classification of finite groups that can be realized as the automorphism group of a polarized abelian surface over a finite field which is maximal in the following sense:

\begin{definition}\label{def 20}
  Let $X$ be an abelian surface over a field $k$, and let $G$ be a finite group. Suppose that the following two conditions hold: \\
\indent (i) there exists an abelian surface $X^{\prime}$ over $k$ that is $k$-isogenous to $X$ with a polarization $\mathcal{L}$ such that $G=\textrm{Aut}_k(X^{\prime},\mathcal{L}),$ and \\
\indent (ii) there is no finite group $H$ such that $G$ is isomorphic to a proper subgroup of $H$ and $H=\textrm{Aut}_k (Y,\mathcal{M})$ for some abelian surface $Y$ over $k$ that is $k$-isogenous to $X$ with a polarization $\mathcal{M}.$ \\

  In this case, $G$ is said to be \emph{realizable maximally (or maximal, in short) in the isogeny class of $X$} as the full automorphism group of a polarized abelian surface over $k$.
\end{definition}

To obtain one of our main result (see Theorem \ref{thm old 24}), we need the following lemmas:
\begin{lemma}\label{lem 21}
  Let $G=Q_8$. If there exists a finite field $k$ and a simple abelian surface $X$ over $k$ such that $G$ is a subgroup of the multiplicative group of $\textrm{End}_k^0(X),$ then $\textrm{End}_k^0(X)=\left( \begin{array}{cc}-1,-1 \\ K \end{array} \right) $ for some totally real quadratic field $K.$
\end{lemma}
\begin{proof}
  Let $k=\F_q$. If $\textrm{dim}_{\Q} \textrm{End}_k^0(X) =4,$ then by Corollary \ref{cor TateEnd0}, $\textrm{End}_k^0(X)$ is a field. Since $G$ is not cyclic, it follows that $\textrm{dim}_{\Q}\textrm{End}_k^0(X)=8$, and hence, $\textrm{End}_k^0(X)$ is a quaternion division algebra over a quadratic number field $K=\Q(\pi_X).$ Since the 2-Sylow subgroup of $G$ is a quaternion group, it follows from \cite[Theorem 9]{1} that $\textrm{End}_k^0(X)=\left( \begin{array}{cc}-1,-1 \\ \Q \end{array} \right) \otimes_{\Q} K.$ \\

  Suppose now that $K$ is an imaginary quadratic number field. Then by \cite[Corollary 2.8]{7}, we know that $a$ is even, and either $\pi_X + \overline{\pi_X} = \pm \sqrt{q}, p \equiv 1~(\textrm{mod}~3)$ or $\pi_X + \overline{\pi_X} =0, p \equiv 1~(\textrm{mod}~4).$ If we are in the first case, then $\Q(\pi_X)=\Q(\sqrt{-3}),$ and if we are in the second case, then $\Q(\pi_X)=\Q(\sqrt{-1}).$ Since $\Q(\sqrt{-1})$ and $\Q(\sqrt{-3})$ split the quaternion algebra $\left( \begin{array}{cc}-1,-1 \\ \Q \end{array} \right) $, we have
  \begin{equation*}
    \textrm{End}_k^0(X)=\left( \begin{array}{cc}-1,-1 \\ \Q \end{array} \right)  \otimes_{\Q} K =M_2(K)
  \end{equation*}
  in both cases, and then since $M_2(K)$ is not a division algebra, this is a contradiction. Hence, we can conclude that $K$ is a totally real quadratic number field. \\

  This completes the proof.
\end{proof}

\begin{lemma}\label{lem 22}
  Let $G=\textrm{Dic}_{16}.$ If there exists a finite field $k$ and a simple abelian surface $X$ over $k$ such that $G$ is a subgroup of the multiplicative group of $\textrm{End}_k^0(X),$ then $\textrm{End}_k^0(X)=\left( \begin{array}{cc}-1,-1 \\ \Q(\sqrt{2}) \end{array} \right).$
\end{lemma}
\begin{proof}
  Let $k=\F_q$. Since the 2-Sylow subgroup of $G$ is a generalized quaternion group of order $16$, as in the proof of Lemma \ref{lem 21}, we can see that $\textrm{End}_k^0(X)=\left( \begin{array}{cc}-1,-1 \\ \Q \end{array} \right) \otimes_{\Q} K$ where $K=\Q(\pi_X)$ is a totally real quadratic number field. Then, since $\textrm{Dic}_{16}$ has an element of order $8$, we can see that $\Q(\zeta_8)$ is a maximal subfield of $\textrm{End}_k^0(X)=\left( \begin{array}{cc}-1,-1 \\ \Q \end{array} \right) \otimes_{\Q} K = \left( \begin{array}{cc}-1,-1 \\ K \end{array} \right).$ In particular, $\Q(\zeta_8)$ is a quadratic extension field of $K,$ and hence, it follows that $K=\Q(\sqrt{2}).$ \\ 

  This completes the proof.
\end{proof}

\begin{lemma}\label{lem 23}
Let $G=\textrm{Dic}_{20}.$ If there exists a finite field $k$ and a simple abelian surface $X$ over $k$ such that $G$ is a subgroup of the multiplicative group of $\textrm{End}_k^0(X),$ then $\textrm{End}_k^0(X)=\left( \begin{array}{cc}-1,-1 \\ \Q(\sqrt{5}) \end{array} \right).$
\end{lemma}
\begin{proof}
  Let $k=\F_q$. As in the proof of Lemma \ref{lem 21}, we can see that $\textrm{dim}_{\Q} \textrm{End}_k^0(X)=8$ and $\textrm{End}_k^0(X)$ is a quaternion division algebra over a quadratic number field $K=\Q(\pi_X).$ Since $\textrm{Dic}_{20}$ has an element of order $10$, we can see that $\Q(\zeta_{10})=\Q(\zeta_5)$ is a maximal subfield of $\textrm{End}_k^0(X).$ In particular, $\Q(\zeta_5)$ is a quadratic extension of $K$, and hence, it follows that $K=\Q(\sqrt{5}).$ Thus, we can see that $\textrm{End}_k^0(X)$ is a quaternion algebra over $\Q(\pi_X)=\Q(\sqrt{5}).$ Then since $\textrm{End}_k^0(X)$ is the unique quaternion algebra over $\Q(\sqrt{5})$ that is ramified at the two infinite places of $\Q(\sqrt{5})$ and split at all finite primes (see \cite[page 528]{15}), it follows that $\textrm{End}_k^0(X)=\left( \begin{array}{cc}-1,-1 \\ \Q(\sqrt{5}) \end{array} \right).$ \\

  This completes the proof.
\end{proof}

Now, we are ready to introduce one of our main results of this section: let $G$ be a finite group.

\begin{theorem}\label{thm old 24}
  There exists a finite field $k$ and a simple abelian surface $X$ over $k$ such that $G$ is the automorphism group of a polarized abelian surface over $k,$ which is maximal in the isogeny class of $X$ if and only if $G$ is one of the following groups (up to isomorphism):
  \begin{center}
  \begin{tabular}{|c|c|}
\hline
$$ & $G$  \\
\hline
$\sharp 1$ & $\Z/2\Z$  \\
\hline
$\sharp 2$ & $\Z/4\Z$  \\
\hline
$\sharp 3$ & $\Z/6\Z$  \\
\hline
$\sharp 4$ & $\Z/8\Z$  \\
\hline
$\sharp 5$ & $\Z/10\Z$  \\
\hline
$\sharp 6$ & $\Z/12\Z$  \\
\hline
$\sharp 7$ & $\textrm{Dic}_{12}$  \\
\hline
$\sharp 8$ & $\mathfrak{T}^*$  \\
\hline
$\sharp 9$ & $\textrm{Dic}_{24}$  \\
\hline
$\sharp 10$ & $\mathfrak{O}^*$  \\
\hline
$\sharp 11$ & $\mathfrak{I}^*$  \\
\hline
\end{tabular}
\vskip 4pt
\textnormal{Table 10}
\end{center}
\end{theorem}

\begin{proof}
  Suppose first that there exists a finite field $k$ and a simple abelian surface $X$ over $k$ such that $G$ is the automorphism group of a polarized abelian surface over $k,$ which is maximal in the isogeny class of $X$. Then by Albert's classification, Corollary \ref{cor 19}, Lemma \ref{lem 21}, Lemma \ref{lem 22}, and Lemma \ref{lem 23}, $G$ is one of the 11 groups in the above table. Hence, it suffices to show the converse. We prove the converse by considering them one by one. \\

  (1) Take $G=\Z/2\Z.$ Let $k=\F_7.$ Then $\pi:=\sqrt{5}+\sqrt{-2}$ is a $7$-Weil number, and hence, by Theorem \ref{thm HondaTata}, there exists a simple abelian variety $X$ over $\F_7$ of dimension $r$ such that $\pi_X$ is conjugate to $\pi.$ Then we have $[\Q(\pi_X):\Q]=[\Q(\pi):\Q]=4$ and $\Q(\pi_X)=\Q(\sqrt{5}+\sqrt{-2})=\Q(\sqrt{5},\sqrt{-2}),$ and hence, $\Q(\pi_X)$ has no real embeddings so that all the local invariants of $\textrm{End}_k^0(X)$ are zero. Hence, it follows from Proposition \ref{index end alg} that $\textrm{End}_k^0(X)=\Q(\pi_X)$ and it is a CM-field of degree $2r$ over $\Q.$ Thus we have $r=2$, and $X$ is a simple abelian surface over $k$ with $\textrm{End}_k^0(X)=\Q(\sqrt{5},\sqrt{-2}).$ Now, let $\mathcal{O}$ be the ring of integers of $\Q(\sqrt{5},\sqrt{-2}).$ (In particular, $\sqrt{-2} \in \mathcal{O}$.) Since $\mathcal{O}$ is a maximal $\Z$-order in $\Q(\sqrt{5},\sqrt{-2}),$ there exists a simple abelian surface $X^{\prime}$ over $k$ such that $X^{\prime}$ is $k$-isogenous to $X$ and $\textrm{End}_k(X^{\prime})=\mathcal{O}$ by Proposition \ref{endposs}. Note also that $\Z/2\Z \cong \langle -1 \rangle \leq \mathcal{O}^{\times}=\textrm{Aut}_k(X^{\prime}).$  \\

  Now, let $\mathcal{L}$ be an ample line bundle on $X^{\prime}$, and put
  \begin{equation*}
    \mathcal{L}^{\prime}:=\bigotimes_{g \in \langle -1 \rangle} g^* \mathcal{L}.
  \end{equation*}
  Then $\mathcal{L}^{\prime}$ is also an ample line bundle on $X^{\prime}$ that is preserved under the action of $\langle -1 \rangle$ so that $\langle -1 \rangle \leq \textrm{Aut}_k(X^{\prime},\mathcal{L}^{\prime}).$ Now, we show that $\langle -1 \rangle$ is a maximal finite subgroup of the multiplicative subgroup of $\Q(\sqrt{5},\sqrt{-2}).$ Indeed, suppose that there is a finite subgroup $H$ of $\Q(\sqrt{5},\sqrt{-2})^{\times}$, which properly contains $\langle -1 \rangle.$ Then $H$ is cyclic so that we may write $H=\langle h \rangle$ for some $h \in H.$ If $|h|=d,$ then we know that $d>2, 2~|~d,$ and $\varphi(d) \leq 4.$ Hence, we get $d \in \{4,6,8,10,12\}.$ If $d=4,$ then $\Q(\zeta_4) \subseteq \Q(\sqrt{5},\sqrt{-2}),$ which is impossible because $\sqrt{-1} \not \in \Q(\sqrt{5},\sqrt{-2}).$ Similarly, if $d=6,$ then we have $\Q(\zeta_6) \subseteq \Q(\sqrt{5},\sqrt{-2})$, which is absurd. Now, if $d=8,$ then we have $\Q(\zeta_8) \subseteq \Q(\sqrt{5},\sqrt{-2}).$ Since both fields have degree 4 over $\Q,$ it follows that $\Q(\zeta_8)=\Q(\sqrt{5}, \sqrt{-2}),$ and this is a contradiction. Thus $d \ne 8.$ Similarly, we can see that $d \ne 10, 12.$ Therefore we can conclude that $\langle -1 \rangle$ is a maximal finite subgroup of $\Q(\sqrt{5},\sqrt{-2})^{\times}$. Then since $\textrm{Aut}_k(X^{\prime},\mathcal{L}^{\prime})$ is a finite subgroup of $\Q(\sqrt{5},\sqrt{-2})^{\times},$ it follows that
  \begin{equation*}
    G \cong \langle -1 \rangle =\textrm{Aut}_k(X^{\prime},\mathcal{L}^{\prime}).
  \end{equation*}

  Now, suppose that $Y$ is an abelian surface over $k$ which is $k$-isogenous to $X.$ In particular, we have $\textrm{End}_k^0(Y)=\Q(\sqrt{5},\sqrt{-2}).$ Suppose that there is a finite group $H$ such that $H=\textrm{Aut}_k(Y, \mathcal{M})$ for a polarization $\mathcal{M}$ on $Y,$ and $G$ is isomorphic to a proper subgroup of $H.$ Then $H$ is a finite subgroup of the multiplicative subgroup of $\Q(\sqrt{5},\sqrt{-2}),$ and hence, $H$ is a cyclic group satisfying $G \cong \langle -1 \rangle \lneq H \leq \Q(\sqrt{5},\sqrt{-2})^{\times}.$ As in the above argument, we can see that this is impossible. Therefore we can conclude that $G$ is maximal in the isogeny class of $X.$ \\

  (2) Take $G=\Z/4\Z.$ Let $k=\F_7$ and let $\pi=\sqrt{6}+\sqrt{-1}$ be a $7$-Weil number. Then a similar argument as in the proof of (1) can be used to show that there is a simple abelian surface $X$ over $k$ such that $G$ is the automorphism group of a polarized abelian surface over $k$, which is maximal in the isogeny class of $X.$ \\ 

  (3) Take $G=\Z/6\Z.$ Let $k=\F_5$ and let $\pi=\sqrt{2}+\sqrt{-3}$ be a $5$-Weil number. Then a similar argument as in the proof of (1) can be used to show that there is a simple abelian surface $X$ over $k$ such that $G$ is the automorphism group of a polarized abelian surface over $k$, which is maximal in the isogeny class of $X.$ \\ 

 (4) Take $G=\Z/8\Z.$ Let $k=\F_4.$ Then $\pi:=\sqrt{4} \cdot \zeta_8 =2 \cdot \zeta_8$ is a $4$-Weil number, and hence, by Theorem \ref{thm HondaTata}, there exists a simple abelian variety $X$ over $\F_4$ of dimension $r$ such that $\pi_X$ is conjugate to $\pi.$ Then we have $[\Q(\pi_X):\Q]=[\Q(\pi):\Q]=4$ and $\Q(\pi_X)=\Q(2 \cdot \zeta_8)=\Q(\zeta_8),$ and hence, $\Q(\pi_X)$ has no real embeddings so that all the local invariants of $\textrm{End}_k^0(X)$ are zero. Hence, it follows from Proposition \ref{index end alg} that $\textrm{End}_k^0(X)=\Q(\pi_X)$ and it is a CM-field of degree $2r$ over $\Q.$ Thus we have $r=2$, and $X$ is a simple abelian surface over $k$ with $\textrm{End}_k^0(X)=\Q(\zeta_8).$ Now, since $\Z[\zeta_8]$ is a (unique) maximal $\Z$-order in $\Q(\zeta_8),$ there exists a simple abelian surface $X^{\prime}$ over $k$ such that $X^{\prime}$ is $k$-isogenous to $X$ and $\textrm{End}_k(X^{\prime})=\Z[\zeta_8]$ by Proposition \ref{endposs}. Note also that $\Z/8\Z \cong \langle \zeta_8 \rangle \leq \Z[\zeta_8]^{\times}=\Z \times \langle \zeta_8 \rangle = \textrm{Aut}_k(X^{\prime}).$  \\

  Now, let $\mathcal{L}$ be an ample line bundle on $X^{\prime}$, and put
  \begin{equation*}
    \mathcal{L}^{\prime}:=\bigotimes_{g \in \langle \zeta_8 \rangle} g^* \mathcal{L}.
  \end{equation*}
  Then $\mathcal{L}^{\prime}$ is also an ample line bundle on $X^{\prime}$ that is preserved under the action of $\langle \zeta_8 \rangle$ so that $\langle \zeta_8 \rangle \leq \textrm{Aut}_k(X^{\prime},\mathcal{L}^{\prime}).$
Note also that the maximal finite subgroup of $\Z \times \langle \zeta_8 \rangle$ is $\langle -1 \rangle \times \langle \zeta_8 \rangle \cong \Z/2\Z \times \Z/8\Z$ by Goursat's Lemma. Since $\textrm{Aut}_k (X^{\prime},\mathcal{L}^{\prime})$ is a finite subgroup of the multiplicative group of $\Q(\zeta_8),$ we know that $\textrm{Aut}_k(X^{\prime}, \mathcal{L}^{\prime})$ is cyclic. It follows that
  \begin{equation*}
    G \cong \langle \zeta_8 \rangle =\textrm{Aut}_k(X^{\prime},\mathcal{L}^{\prime}).
  \end{equation*}

  Now, suppose that $Y$ is an abelian surface over $k$ which is $k$-isogenous to $X.$ In particular, we have $\textrm{End}_k^0(Y)=\Q(\zeta_8).$ Suppose that there is a finite group $H$ such that $H=\textrm{Aut}_k(Y, \mathcal{M})$ for a polarization $\mathcal{M}$ on $Y,$ and $G$ is isomorphic to a proper subgroup of $H.$ Then $H$ is a finite subgroup of the multiplicative subgroup of $\Q(\zeta_8),$ and hence, $H$ is a cyclic group satisfying $G \cong \langle \zeta_8 \rangle \lneq H \leq \langle -1 \rangle \times \langle \zeta_8 \rangle.$ It follows from the second condition that $H =\langle -1 \rangle \times \langle \zeta_8 \rangle$, and this contradicts the fact that $H$ is cyclic. Therefore we can conclude that $G$ is maximal in the isogeny class of $X.$ \\

(5) Take $G=\Z/10\Z.$ Let $k=\F_{25}$ and let $\pi=\sqrt{25}\cdot \zeta_{10}=5 \cdot \zeta_{10}$ be a $25$-Weil number. Then a similar argument as in the proof of (4) can be used to show that there is a simple abelian surface $X$ over $k$ such that $G$ is the automorphism group of a polarized abelian surface over $k$, which is maximal in the isogeny class of $X.$ \\ 

(6) Take $G=\Z/12\Z.$ Let $k=\F_{9}$ and let $\pi=\sqrt{9}\cdot \zeta_{12}=3 \cdot \zeta_{12}$ be a $9$-Weil number. Then a similar argument as in the proof of (4) can be used to show that there is a simple abelian surface $X$ over $k$ such that $G$ is the automorphism group of a polarized abelian surface over $k$, which is maximal in the isogeny class of $X.$ \\ 

(7) Take $G=\textrm{Dic}_{12}.$ Let $k=\F_{11}.$ Then $\pi:=\sqrt{11}$ is a $11$-Weil number, and hence, by Theorem \ref{thm HondaTata}, there exists a simple abelian variety $X$ over $\F_{11}$ of dimension $r$ such that $\pi_X$ is conjugate to $\pi.$ Hence, by Remark \ref{Qpi Real}, we have $r=2$, and $\textrm{End}_k^0(X)$ is the unique quaternion algebra over $\Q(\sqrt{11})$ that is ramified at two infinite (real) places and split at all finite primes of $\Q(\sqrt{11})$ (which equals $\mathcal{A}$ below). \\

Let $\mathcal{A}=\left( \begin{array}{cc}-1,-3 \\ \Q(\sqrt{11}) \end{array} \right), K=\Q(\sqrt{11}),$ and $R=\Z[\sqrt{11}]$ (which is the ring of integers of $K$ so that it is a maximal $\Z$-order in $K$). Let $\mathcal{O}$ be an $R$-submodule of $\mathcal{A}$ generated by $\left \{1, i, \frac{1+j}{2}, \frac{i+ij}{2} \right\},$ and let $\mathcal{O}^{\prime}$ be a maximal $\Z$-order containing $\mathcal{O}$ in $\mathcal{A}.$ For the ease of notation, let $\alpha = \frac{1+j}{2}$ and $\beta = i.$ Then the group $G:=\langle \alpha, \beta \rangle$ is isomorphic to $\textrm{Dic}_{12}$ and it is easy to check that every element of $G$ is a unit in $\mathcal{O}^{\prime}.$ \\

Now, since $\mathcal{O}^{\prime}$ is a maximal $\Z$-order in $\mathcal{A},$ there exists a simple abelian surface $X^{\prime}$ over $k$ such that $X^{\prime}$ is $k$-isogenous to $X$ and $\textrm{End}_k(X^{\prime})=\mathcal{O}^{\prime}$ by Proposition \ref{endposs}. Note also that $\textrm{Dic}_{12} \cong G \leq (\mathcal{O}^{\prime})^{\times}=\textrm{Aut}_k(X^{\prime})$ (by the last statement above).  \\

  Now, let $\mathcal{L}$ be an ample line bundle on $X^{\prime}$, and put
  \begin{equation*}
    \mathcal{L}^{\prime}:=\bigotimes_{g \in G} g^* \mathcal{L}.
  \end{equation*}
  Then $\mathcal{L}^{\prime}$ is also an ample line bundle on $X^{\prime}$ that is preserved under the action of $G$ so that $G \leq \textrm{Aut}_k(X^{\prime},\mathcal{L}^{\prime}).$ Since $\textrm{Aut}_k (X^{\prime},\mathcal{L}^{\prime})$ is a finite subgroup of the multiplicative group of $\mathcal{A}=\textrm{End}_k^0(X^{\prime}),$ it follows from Corollary \ref{cor 19} (together with the argument below) that
  \begin{equation*}
    \textrm{Dic}_{12} \cong G = \textrm{Aut}_k(X^{\prime}, \mathcal{L}^{\prime}).
  \end{equation*}

  Now, suppose that $Y$ is an abelian surface over $k$ which is $k$-isogenous to $X.$ In particular, we have $\textrm{End}_k^0(Y)=\left( \begin{array}{cc}-1,-3 \\ \Q(\sqrt{11}) \end{array} \right).$ Suppose that there is a finite group $H$ such that $H=\textrm{Aut}_k(Y, \mathcal{M})$ for a polarization $\mathcal{M}$ on $Y,$ and $G$ is isomorphic to a proper subgroup of $H.$ Then $H$ is a finite subgroup of the multiplicative subgroup of $\textrm{End}_k^0(Y)=\left( \begin{array}{cc}-1,-3 \\ \Q(\sqrt{11}) \end{array} \right),$ and hence, it follows from Corollary \ref{cor 19} that $H=\textrm{Dic}_{24}.$ Since $H$ contains an element of order $12,$ the center of $\left( \begin{array}{cc}-1,-3 \\ \Q(\sqrt{11}) \end{array} \right)$ should contain $\zeta_{12} + \zeta_{12}^{-1} = \sqrt{3}$, and this is a contradiction. Therefore we can conclude that $G$ is maximal in the isogeny class of $X.$ \\

(8) Take $G=\mathfrak{T}^*.$ Let $k=\F_{3}.$ Then $\pi:=\sqrt{3}$ is a $3$-Weil number, and hence, by Theorem \ref{thm HondaTata}, there exists a simple abelian variety $X$ over $\F_{3}$ of dimension $r$ such that $\pi_X$ is conjugate to $\pi.$ Hence, by Remark \ref{Qpi Real}, we have $r=2$, and $\textrm{End}_k^0(X)$ is the unique quaternion algebra over $\Q(\sqrt{3})$ that is ramified at two infinite (real) places and split at all finite primes of $\Q(\sqrt{3})$ (which equals $\mathcal{A}$ below). \\

Let $\mathcal{A}=\left( \begin{array}{cc}-1,-1 \\ \Q(\sqrt{3}) \end{array} \right), K=\Q(\sqrt{3}),$ and $R=\Z[\sqrt{3}]$ (which is the ring of integers of $K$ so that it is a maximal $\Z$-order in $K$). Let $\mathcal{O}$ be an $R$-submodule of $\mathcal{A}$ generated by $\left \{i, j, ij, \frac{1+i+j+ij}{2} \right\},$ and let $\mathcal{O}^{\prime}$ be a maximal $\Z$-order containing $\mathcal{O}$ in $\mathcal{A}.$ Let $G=\{\pm 1, \pm i, \pm j, \pm ij, \frac{1}{2}(\pm 1 \pm i \pm j \pm ij) \} = \mathfrak{T}^*.$ It is easy to check that every element of $G$ is a unit in $\mathcal{O}^{\prime}.$ \\

Now, since $\mathcal{O}^{\prime}$ is a maximal $\Z$-order in $\mathcal{A},$ there exists a simple abelian surface $X^{\prime}$ over $k$ such that $X^{\prime}$ is $k$-isogenous to $X$ and $\textrm{End}_k(X^{\prime})=\mathcal{O}^{\prime}$ by Proposition \ref{endposs}. Note also that $G \leq (\mathcal{O}^{\prime})^{\times}=\textrm{Aut}_k(X^{\prime})$ (by the last statement above).  \\

  Now, let $\mathcal{L}$ be an ample line bundle on $X^{\prime}$, and put
  \begin{equation*}
    \mathcal{L}^{\prime}:=\bigotimes_{g \in G} g^* \mathcal{L}.
  \end{equation*}
  Then $\mathcal{L}^{\prime}$ is also an ample line bundle on $X^{\prime}$ that is preserved under the action of $G$ so that $G \leq \textrm{Aut}_k(X^{\prime},\mathcal{L}^{\prime}).$ Since $\textrm{Aut}_k (X^{\prime},\mathcal{L}^{\prime})$ is a finite subgroup of the multiplicative group of $\mathcal{A}=\textrm{End}_k^0(X^{\prime}),$ it follows from Corollary \ref{cor 19} that
  \begin{equation*}
    \mathfrak{T}^* = G = \textrm{Aut}_k(X^{\prime}, \mathcal{L}^{\prime}).
  \end{equation*}
  
Now, suppose that $Y$ is an abelian surface over $k$ which is $k$-isogenous to $X.$ In particular, we have $\textrm{End}_k^0(Y)=\left( \begin{array}{cc}-1,-1 \\ \Q(\sqrt{3}) \end{array} \right).$ Suppose that there is a finite group $H$ such that $H=\textrm{Aut}_k(Y, \mathcal{M})$ for a polarization $\mathcal{M}$ on $Y,$ and $G$ is isomorphic to a proper subgroup of $H.$ Then $H$ is a finite subgroup of the multiplicative subgroup of $\textrm{End}_k^0(Y)=\left( \begin{array}{cc}-1,-1 \\ \Q(\sqrt{3}) \end{array} \right),$ and hence, it follows from Corollary \ref{cor 19} that $H=G$, which is a contradiction. Therefore we can conclude that $G$ is maximal in the isogeny class of $X.$ \\

(9) Take $G=\textrm{Dic}_{24}.$ Let $k=\F_{3}.$ Then $\pi:=\sqrt{3}$ is a $3$-Weil number, and hence, by Theorem \ref{thm HondaTata}, there exists a simple abelian variety $X$ over $\F_{3}$ of dimension $r$ such that $\pi_X$ is conjugate to $\pi.$ Hence, by Remark \ref{Qpi Real}, we have $r=2$, and $\textrm{End}_k^0(X)$ is the unique quaternion algebra over $\Q(\sqrt{3})$ that is ramified at two infinite (real) places and split at all finite primes of $\Q(\sqrt{3})$ (which equals $\mathcal{A}$ below). \\

Let $\mathcal{A}=\left( \begin{array}{cc}-1,-1 \\ \Q(\sqrt{3}) \end{array} \right), K=\Q(\sqrt{3}),$ and $R=\Z[\sqrt{3}]$. Let $\mathcal{O}$ be an $R$-submodule of $\mathcal{A}$ generated by $\left \{1, i, \frac{\sqrt{3}+j}{2}, \frac{\sqrt{3}+\sqrt{3}i +j+ij}{2} \right\}.$ According to an easy computation of discriminant, $\mathcal{O}$ is a maximal $\Z$-order in $\mathcal{A}$ having the discriminant $R.$ For the ease of notation, let $\alpha = \frac{\sqrt{3}+j}{2}$ and $\beta = i.$ Then the group $G:=\langle \alpha, \beta \rangle$ is isomorphic to $\textrm{Dic}_{24}$ and it is easy to check that every element of $G$ is a unit in $\mathcal{O}.$ \\

Now, since $\mathcal{O}$ is a maximal $\Z$-order in $\mathcal{A},$ there exists a simple abelian surface $X^{\prime}$ over $k$ such that $X^{\prime}$ is $k$-isogenous to $X$ and $\textrm{End}_k(X^{\prime})=\mathcal{O}$ by Proposition \ref{endposs}. Note also that $\textrm{Dic}_{24} \cong G \leq \mathcal{O}^{\times}=\textrm{Aut}_k(X^{\prime})$ (by the last statement above).  \\

  Now, let $\mathcal{L}$ be an ample line bundle on $X^{\prime}$, and put
  \begin{equation*}
    \mathcal{L}^{\prime}:=\bigotimes_{g \in G} g^* \mathcal{L}.
  \end{equation*}
  Then $\mathcal{L}^{\prime}$ is also an ample line bundle on $X^{\prime}$ that is preserved under the action of $G$ so that $G \leq \textrm{Aut}_k(X^{\prime},\mathcal{L}^{\prime}).$ Since $\textrm{Aut}_k (X^{\prime},\mathcal{L}^{\prime})$ is a finite subgroup of the multiplicative group of $\mathcal{A}=\textrm{End}_k^0(X^{\prime}),$ it follows from Corollary \ref{cor 19} that
  \begin{equation*}
    \textrm{Dic}_{24} \cong G = \textrm{Aut}_k(X^{\prime}, \mathcal{L}^{\prime}).
  \end{equation*}

  Now, suppose that $Y$ is an abelian surface over $k$ which is $k$-isogenous to $X.$ In particular, we have $\textrm{End}_k^0(Y)=\left( \begin{array}{cc}-1,-1 \\ \Q(\sqrt{3}) \end{array} \right).$ Suppose that there is a finite group $H$ such that $H=\textrm{Aut}_k(Y, \mathcal{M})$ for a polarization $\mathcal{M}$ on $Y,$ and $G$ is isomorphic to a proper subgroup of $H.$ Then $H$ is a finite subgroup of the multiplicative subgroup of $\textrm{End}_k^0(Y)=\left( \begin{array}{cc}-1,-1 \\ \Q(\sqrt{3}) \end{array} \right),$ and hence, it follows from Corollary \ref{cor 19} that $H=G$, which is a contradiction. Therefore we can conclude that $G$ is maximal in the isogeny class of $X.$ \\

(10) Take $G=\mathfrak{O}^*.$ Let $k=\F_{2}.$ Then $\pi:=\sqrt{2}$ is a $2$-Weil number, and hence, by Theorem \ref{thm HondaTata}, there exists a simple abelian variety $X$ over $\F_{2}$ of dimension $r$ such that $\pi_X$ is conjugate to $\pi.$ Hence, by Remark \ref{Qpi Real}, we have $r=2$, and $\textrm{End}_k^0(X)$ is the unique quaternion algebra over $\Q(\sqrt{2})$ that is ramified at two infinite (real) places and split at all finite primes of $\Q(\sqrt{2})$ (which equals $\mathcal{A}$ below). \\

Let $\mathcal{A}=\left( \begin{array}{cc}-1,-1 \\ \Q(\sqrt{2}) \end{array} \right), K=\Q(\sqrt{2}),$ and $R=\Z[\sqrt{2}]$ (which is the ring of integers of $K$ so that it is a maximal $\Z$-order in $K$). Let $\mathcal{O}$ be an $R$-submodule of $\mathcal{A}$ generated by $\left \{1, \frac{\sqrt{2}}{2}(1+i), \frac{\sqrt{2}}{2}(1+j), \frac{1+i+j+ij}{2} \right\}.$ According to an easy computation of discriminant, $\mathcal{O}$ is a maximal $\Z$-order in $\mathcal{A}$ having the discriminant $R.$ For the ease of notation, let $\alpha = \frac{\sqrt{2}}{2}(1+j), \beta =i, $ and $\gamma = \frac{-1+i+j+ij}{2}.$ Then the group $G:=\langle \alpha, \beta, \gamma \rangle$ is isomorphic to $\mathfrak{O}^*$ and it is easy to check that every element of $G$ is a unit in $\mathcal{O}.$  \\

Then a similar argument as in the proof of (9) can be used to show that there is a simple abelian surface over $k$ such that $G$ is the automorphism group of a polarized abelian surface over $k,$ which is maximal in the isogeny class of $X.$ \\

(11) Take $G=\mathfrak{I}^*.$ Let $k=\F_{5}.$ Then $\pi:=\sqrt{5}$ is a $5$-Weil number, and hence, by Theorem \ref{thm HondaTata}, there exists a simple abelian variety $X$ over $\F_{5}$ of dimension $r$ such that $\pi_X$ is conjugate to $\pi.$ Hence, by Remark \ref{Qpi Real}, we have $r=2$, and $\textrm{End}_k^0(X)$ is the unique quaternion algebra over $\Q(\sqrt{5})$ that is ramified at two infinite (real) places and split at all finite primes of $\Q(\sqrt{5})$ (which equals $\mathcal{A}$ below). \\

Let $\mathcal{A}=\left( \begin{array}{cc}-1,-1 \\ \Q(\sqrt{5}) \end{array} \right), K=\Q(\sqrt{5}),$ and $R=\Z \left[ \frac{1+\sqrt{5}}{2}\right]$ (which is the ring of integers of $K$ so that it is a maximal $\Z$-order in $K$). Let $\mathcal{O}$ be an $R$-submodule of $\mathcal{A}$ generated by $\left \{1, i, \frac{\lambda +\lambda^{-1} i + j}{2}, \frac{\lambda i - \lambda^{-1}+ ij}{2}\right\}$ where $\lambda = \frac{1+\sqrt{5}}{2}.$ According to an easy computation of discriminant, $\mathcal{O}$ is a maximal $\Z$-order in $\mathcal{A}$ having the discriminant $R.$ It is also easy to check that the group of units of $\mathcal{O}$ contains a subgroup $G$ which is isomorphic to $\mathfrak{I}^*.$ \\

Then a similar argument as in the proof of (9) can be used to show that there is a simple abelian surface over $k$ such that $G$ is the automorphism group of a polarized abelian surface over $k,$ which is maximal in the isogeny class of $X.$ \\
 
 This completes the proof.
\end{proof}

Our next result takes care of the first case of non-simple abelian surface:

\begin{theorem}\label{prodnonisoellip}
There exist a finite field $k$ and two non-isogenous elliptic curves $E_1$ and $E_2$ over $k$ such that $G$ is the automorphism group of a polarized abelian surface over $k,$ which is maximal in the isogeny class of $X:=E_1 \times E_2$ if and only if $G$ is one of the following groups (up to isomorphism): 
  \begin{center}
  \begin{tabular}{|c|c|}
\hline
$$ & $G$  \\
\hline
$\sharp 1$ & $\Z/2\Z \times \Z/2\Z$  \\
\hline
$\sharp 2$ & $\Z/2\Z \times \Z/4\Z$  \\
\hline
$\sharp 3$ & $\Z/2\Z \times \Z/6\Z $  \\
\hline
$\sharp 4$ & $\Z/2\Z \times \textrm{Dic}_{12} $  \\
\hline
$\sharp 5$ & $\Z/2\Z \times \mathfrak{T}^* $  \\
\hline
$\sharp 6$ & $\Z/4\Z \times \Z/4\Z$  \\
\hline
$\sharp 7$ & $\Z/4\Z \times \Z/6\Z$  \\
\hline
$\sharp 8$ & $\Z/4\Z \times \textrm{Dic}_{12}$  \\
\hline
$\sharp 9$ & $\Z/4\Z \times \mathfrak{T}^*$  \\
\hline
$\sharp 10$ & $\Z/6\Z \times \Z/6\Z$  \\
\hline
$\sharp 11$ & $\Z/6\Z \times \textrm{Dic}_{12}$  \\
\hline
$\sharp 12$ & $\Z/6\Z \times \mathfrak{T}^*$  \\
\hline
$\sharp 13$ & $\textrm{Dic}_{12} \times \textrm{Dic}_{12}$  \\
\hline
$\sharp 14$ & $\mathfrak{T}^* \times \mathfrak{T}^*$  \\
\hline
\end{tabular}
\vskip 4pt
\textnormal{Table 11}
\end{center}
\end{theorem}
\begin{proof}
Suppose first that there exist a finite field $k$ and two non-isogenous elliptic curves $E_1$ and $E_2$ over $k$ such that $G$ is the automorphism group of a polarized abelian surface over $k$, which is maximal in the isogeny class of $X:=E_1 \times E_2.$ In particular, we have $\textrm{End}_k^0(X)=\Q(\sqrt{-d_1})\oplus \Q(\sqrt{-d_2})$ for some square-free positive integers $d_1$ and $d_2$ or $\textrm{End}_k^0(X)=\Q(\sqrt{-d})\oplus D_{p,\infty}$ for some square-free positive integer $d$ and a prime number $p,$ or $\textrm{End}_k^0(X)=D_{p,\infty} \oplus D_{p,\infty}$ for some prime number $p$ in view of Theorem \ref{isogclass ell}. Since $G$ is a maximal finite subgroup of $\textrm{End}_k^0(X)$ by assumption, it follows from Goursat's Lemma that $G$ must be one of the 14 groups in the above table. Hence, it suffices to show the converse. We prove the converse by considering them one by one. \\

(1) Take $G=\Z/2\Z \times \Z/2\Z.$ Let $k=\F_3.$ Then there are two non-isogenous ordinary elliptic curves $E_1, E_2$ over $k$ such that $\textrm{End}_k^0(E_1)=\Q(\sqrt{-2})$ and $\textrm{End}_k^0(E_2)=\Q(\sqrt{-11})$ by Theorem \ref{isogclass ell}. Let $X=E_1 \times E_2$ so that $\textrm{End}_k^0(X)=\Q(\sqrt{-2})\oplus \Q(\sqrt{-11}).$ By Theorem \ref{max gen}, $\mathcal{O}:=\Z[\sqrt{-2}]\oplus \Z \left[\frac{1+\sqrt{-11}}{2}\right]$ is a maximal $\Z$-order in $\Q(\sqrt{-2})\oplus \Q(\sqrt{-11})$, and hence, there exists an abelian surface $X^{\prime}$ over $k$ such that $X^{\prime}$ is $k$-isogenous to $X$ and $\textrm{End}_k(X^{\prime})=\mathcal{O}$ by Proposition \ref{endposs}. Then it follows that $G \cong \mathcal{O}^{\times}=\textrm{Aut}_k(X^{\prime})$ (with a canonical polarization). \\

  (2) Take $G=\Z/2\Z \times \Z/4\Z.$ Let $k=\F_5.$ Then there are two non-isogenous ordinary elliptic curves $E_1, E_2$ over $k$ such that $\textrm{End}_k^0(E_1)=\Q(\sqrt{-11})$ and $\textrm{End}_k^0(E_2)=\Q(\sqrt{-1})$ by Theorem \ref{isogclass ell}. Let $X=E_1 \times E_2$ so that $\textrm{End}_k^0(X)=\Q(\sqrt{-11})\oplus \Q(\sqrt{-1}).$ By Theorem \ref{max gen}, $\mathcal{O}:=\Z \left[\frac{1+\sqrt{-11}}{2}\right] \oplus \Z[\sqrt{-1}] $ is a maximal $\Z$-order in $\Q(\sqrt{-11})\oplus \Q(\sqrt{-1}).$ Then a similar argument as in the proof of (1) can be used to show that there exists an abelian surface $X^{\prime}$ over $k$ that is $k$-isogenous to $X$ such that $G \cong \textrm{Aut}_k(X^{\prime})$ (with a canonical polarization). \\

  (3) Take $G=\Z/2\Z \times \Z/6\Z.$ Let $k=\F_7.$ Then there are two non-isogenous ordinary elliptic curves $E_1, E_2$ over $k$ such that $\textrm{End}_k^0(E_1)=\Q(\sqrt{-6})$ and $\textrm{End}_k^0(E_2)=\Q(\sqrt{-3})$ by Theorem \ref{isogclass ell}. Let $X=E_1 \times E_2$ so that $\textrm{End}_k^0(X)=\Q(\sqrt{-6})\oplus \Q(\sqrt{-3}).$ By Theorem \ref{max gen}, $\mathcal{O}:= \Z[\sqrt{-6}] \oplus \Z \left[\frac{1+\sqrt{-3}}{2}\right] $ is a maximal $\Z$-order in $\Q(\sqrt{-6})\oplus \Q(\sqrt{-3}).$ Then a similar argument as in the proof of (1) can be used to show that there exists an abelian surface $X^{\prime}$ over $k$ that is $k$-isogenous to $X$ such that $G \cong \textrm{Aut}_k(X^{\prime})$ (with a canonical polarization). \\

(4) Take $G=\Z/2\Z \times \textrm{Dic}_{12}.$ Let $k=\F_{9}.$ Then there is an ordinary elliptic curve $E_1$ over $k$ with $\textrm{End}_k^0(E_1)=\textrm{End}_{\F_{3}}^0(E_1)=\Q(\sqrt{-2})$ by Theorem \ref{isogclass ell}. Let $E_2$ be a supersingular elliptic curve over $k$ (all of whose endomorphisms are defined over $k$) so that we have $\textrm{End}_{k}^0(E_2)=D_{3,\infty}=\left( \begin{array}{cc} -1, -3 \\ \Q  \end{array} \right).$ Let $X=E_1 \times E_2.$ Then we have $D:=\textrm{End}_k^0(X)=\Q(\sqrt{-2}) \oplus D_{3,\infty}.$ Let $\mathcal{O}$ be a (unique) maximal $\Z$-order in $D_{3, \infty}$ (that is generated by $1, i, \frac{1+j}{2}, \frac{i+ij}{2})$, and let $\mathcal{O}^{\prime}=\Z [\sqrt{-2}] \oplus \mathcal{O}.$ By Theorem \ref{max gen}, $\mathcal{O}^{\prime}$ is a maximal $\Z$-order in $\Q(\sqrt{-2})\oplus D_{3,\infty}$, and hence, there exists an abelian surface $X^{\prime}$ over $k$ such that $X^{\prime}$ is $k$-isogenous to $X$ and $\textrm{End}_k(X^{\prime})=\mathcal{O}^{\prime}$ by Proposition \ref{endposs}. Then it follows that $G \cong (\mathcal{O}^{\prime})^{\times}=\textrm{Aut}_k(X^{\prime})$ (with a canonical polarization). \\

(5) Take $G=\Z/2\Z \times \mathfrak{T}^*.$ Let $k=\F_{4}.$ Then there is an ordinary elliptic curve $E_1$ over $k$ with $\textrm{End}_k^0(E_1)=\textrm{End}_{\F_{2}}^0(E_1)=\Q(\sqrt{-7})$ by Theorem \ref{isogclass ell}. Let $E_2$ be a supersingular elliptic curve over $k$ (all of whose endomorphisms are defined over $k$) so that we have $\textrm{End}_{k}^0(E_2)=D_{2,\infty}=\left( \begin{array}{cc} -1, -1 \\ \Q  \end{array} \right).$ Let $X=E_1 \times E_2.$ Then we have $D:=\textrm{End}_k^0(X)=\Q(\sqrt{-7}) \oplus D_{2,\infty}.$ Let $\mathcal{O}$ be a (unique) maximal $\Z$-order in $D_{2, \infty}$ (that is generated by $i, j, ij, \frac{1+i+j+ij}{2})$, and let $\mathcal{O}^{\prime}=\Z \left[\frac{1+\sqrt{-7}}{2}\right] \oplus \mathcal{O}.$ By Theorem \ref{max gen}, $\mathcal{O}^{\prime}$ is a maximal $\Z$-order in $\Q(\sqrt{-7})\oplus D_{2,\infty}.$ Then a similar argument as in the proof of (4) can be used to show that there exists an abelian surface $X^{\prime}$ over $k$ that is $k$-isogenous to $X$ such that $G \cong \textrm{Aut}_k(X^{\prime})$ (with a canonical polarization). \\

(6) Take $G=\Z/4\Z \times \Z/4\Z.$ Let $k=\F_5.$ Then there are two non-isogenous ordinary elliptic curves $E_1, E_2$ over $k$ such that $\textrm{End}_k^0(E_1)=\textrm{End}_k^0(E_2)=\Q(\sqrt{-1})$ by Theorem \ref{isogclass ell}. Let $X=E_1 \times E_2$ so that $\textrm{End}_k^0(X)=\Q(\sqrt{-1})\oplus \Q(\sqrt{-1}).$ By Theorem \ref{max gen}, $\mathcal{O}:= \Z[\sqrt{-1}] \oplus \Z[\sqrt{-1}] $ is a maximal $\Z$-order in $\Q(\sqrt{-1})\oplus \Q(\sqrt{-1}).$ Then a similar argument as in the proof of (1) can be used to show that there exists an abelian surface $X^{\prime}$ over $k$ that is $k$-isogenous to $X$ such that $G \cong \textrm{Aut}_k(X^{\prime})$ (with a canonical polarization). \\

(7) Take $G=\Z/4\Z \times \Z/6\Z.$ Let $k=\F_{13}.$ Then there are two non-isogenous ordinary elliptic curves $E_1, E_2$ over $k$ such that $\textrm{End}_k^0(E_1)=\Q(\sqrt{-1})$ and $\textrm{End}_k^0(E_2)=\Q(\sqrt{-3})$ by Theorem \ref{isogclass ell}. Let $X=E_1 \times E_2$ so that $\textrm{End}_k^0(X)=\Q(\sqrt{-1})\oplus \Q(\sqrt{-3}).$ By Theorem \ref{max gen}, $\mathcal{O}:= \Z[\sqrt{-1}] \oplus \Z \left[\frac{1+\sqrt{-3}}{2}\right] $ is a maximal $\Z$-order in $\Q(\sqrt{-1})\oplus \Q(\sqrt{-3}).$ Then a similar argument as in the proof of (1) can be used to show that there exists an abelian surface $X^{\prime}$ over $k$ that is $k$-isogenous to $X$ such that $G \cong \textrm{Aut}_k(X^{\prime})$ (with a canonical polarization). \\

(8) Take $G=\Z/4\Z \times \textrm{Dic}_{12}.$ Let $k=\F_{9}.$ Then there is a supersingular elliptic curve $E_1$ over $k$ with $\textrm{End}_k^0(E_1)=\Q(\sqrt{-1})$ by Theorem \ref{isogclass ell}. Let $E_2$ be a supersingular elliptic curve over $k$ (all of whose endomorphisms are defined over $k$) so that we have $\textrm{End}_{k}^0(E_2)=D_{3,\infty}.$ Let $X=E_1 \times E_2.$ Then a similar argument as in the proof of (4) can be used to show that there exists an abelian surface $X^{\prime}$ over $k$ that is $k$-isogenous to $X$ such that $G \cong \textrm{Aut}_k(X^{\prime})$ (with a canonical polarization). \\

(9) Take $G=\Z/4\Z \times \mathfrak{T}^*.$ Let $k=\F_{4}.$ Then there is a supersingular elliptic curve $E_1$ over $k$ with $\textrm{End}_k^0(E_1)=\Q(\sqrt{-1})$ by Theorem \ref{isogclass ell}. Let $E_2$ be a supersingular elliptic curve over $k$ (all of whose endomorphisms are defined over $k$) so that we have $\textrm{End}_{k}^0(E_2)=D_{2,\infty}.$ Let $X=E_1 \times E_2.$ Then a similar argument as in the proof of (5) can be used to show that there exists an abelian surface $X^{\prime}$ over $k$ that is $k$-isogenous to $X$ such that $G \cong \textrm{Aut}_k(X^{\prime})$ (with a canonical polarization). \\

(10) Take $G=\Z/6\Z \times \Z/6\Z.$ Let $k=\F_{7}.$ Then there are two non-isogenous ordinary elliptic curves $E_1, E_2$ over $k$ such that $\textrm{End}_k^0(E_1)=\textrm{End}_k^0(E_2)=\Q(\sqrt{-3})$ by Theorem \ref{isogclass ell}. Let $X=E_1 \times E_2$ so that $\textrm{End}_k^0(X)=\Q(\sqrt{-3})\oplus \Q(\sqrt{-3}).$ By Theorem \ref{max gen}, $\mathcal{O}:= \Z \left[\frac{1+\sqrt{-3}}{2}\right] \oplus \Z \left[\frac{1+\sqrt{-3}}{2}\right] $ is a maximal $\Z$-order in $\Q(\sqrt{-3})\oplus \Q(\sqrt{-3}).$ Then a similar argument as in the proof of (1) can be used to show that there exists an abelian surface $X^{\prime}$ over $k$ that is $k$-isogenous to $X$ such that $G \cong \textrm{Aut}_k(X^{\prime})$ (with a canonical polarization). \\

(11) Take $G=\Z/6\Z \times \textrm{Dic}_{12}.$ Let $k=\F_{9}.$ Then there is a supersingular elliptic curve $E_1$ over $k$ with $\textrm{End}_k^0(E_1)=\Q(\sqrt{-3})$ by Theorem \ref{isogclass ell}. Let $E_2$ be a supersingular elliptic curve over $k$ (all of whose endomorphisms are defined over $k$) so that we have $\textrm{End}_{k}^0(E_2)=D_{3,\infty}.$ Let $X=E_1 \times E_2.$ Then a similar argument as in the proof of (4) can be used to show that there exists an abelian surface $X^{\prime}$ over $k$ that is $k$-isogenous to $X$ such that $G \cong \textrm{Aut}_k(X^{\prime})$ (with a canonical polarization). \\

(12) Take $G=\Z/6\Z \times \mathfrak{T}^*.$ Let $k=\F_{4}.$ Then there is a supersingular elliptic curve $E_1$ over $k$ with $\textrm{End}_k^0(E_1)=\Q(\sqrt{-3})$ by Theorem \ref{isogclass ell}. Let $E_2$ be a supersingular elliptic curve over $k$ (all of whose endomorphisms are defined over $k$) so that we have $\textrm{End}_{k}^0(E_2)=D_{2,\infty}.$ Let $X=E_1 \times E_2.$ Then a similar argument as in the proof of (5) can be used to show that there exists an abelian surface $X^{\prime}$ over $k$ that is $k$-isogenous to $X$ such that $G \cong \textrm{Aut}_k(X^{\prime})$ (with a canonical polarization). \\

(13) Take $G=\textrm{Dic}_{12} \times \textrm{Dic}_{12}.$ Let $k=\F_{9}.$ Then there are two non-isogenous supersingular elliptic curves $E_1, E_2$ over $k$ with $\textrm{End}_k^0(E_1)=\textrm{End}_k^0(E_2)=D_{3,\infty}$ by Theorem \ref{isogclass ell}. Let $X=E_1 \times E_2.$ Then we have $D:=\textrm{End}_k^0(X)= D_{3,\infty} \oplus D_{3,\infty}.$ Let $\mathcal{O}$ be a (unique) maximal $\Z$-order in $D_{3, \infty}$ (that is generated by $1, i, \frac{1+j}{2}, \frac{i+ij}{2})$, and let $\mathcal{O}^{\prime}=\mathcal{O} \oplus \mathcal{O}.$ By Theorem \ref{max gen}, $\mathcal{O}^{\prime}$ is a maximal $\Z$-order in $D_{3,\infty} \oplus D_{3,\infty}$, and hence, there exists an abelian surface $X^{\prime}$ over $k$ such that $X^{\prime}$ is $k$-isogenous to $X$ and $\textrm{End}_k(X^{\prime})=\mathcal{O}^{\prime}$ by Proposition \ref{endposs}. Then it follows that $G \cong (\mathcal{O}^{\prime})^{\times}=\textrm{Aut}_k(X^{\prime})$ (with a canonical polarization). \\

(14) Take $G=\mathfrak{T}^* \times  \mathfrak{T}^*.$ Let $k=\F_{4}.$ Then there are two non-isogenous supersingular elliptic curves $E_1, E_2$ over $k$ with $\textrm{End}_k^0(E_1)=\textrm{End}_k^0(E_2)=D_{2,\infty}$ by Theorem \ref{isogclass ell}. Let $X=E_1 \times E_2.$ Then we have $D:=\textrm{End}_k^0(X)= D_{2,\infty} \oplus D_{2,\infty}.$ Let $\mathcal{O}$ be a (unique) maximal $\Z$-order in $D_{2, \infty}$ (that is generated by $i, j, ij, \frac{1+i+j+ij}{2})$, and let $\mathcal{O}^{\prime}=\mathcal{O} \oplus \mathcal{O}.$ By Theorem \ref{max gen}, $\mathcal{O}^{\prime}$ is a maximal $\Z$-order in $D_{2,\infty} \oplus D_{2,\infty}.$ Then a similar argument as in the proof of (13) can be used to show that there exists an abelian surface $X^{\prime}$ over $k$ that is $k$-isogenous to $X$ such that $G \cong \textrm{Aut}_k(X^{\prime})$ (with a canonical polarization). \\

This completes the proof.
\end{proof}

In order to move to the next case, we introduce the following
\begin{lemma}\label{sim to Lange lem}
  Let $E$ be an ordinary elliptic curve over a finite field $k,$ and let $X=E^2.$ Suppose that $G$ is a finite subgroup of $\textrm{Aut}_k (X).$ Then we have: \\
  (a) If $f \in G$ is of order $e,$ then $e \in \{2,3,4,6,8, 12\}.$ \\
  (b) $G$ is a solvable group of order $2^m 3^n$ for some non-negative integers $m$ and $n.$
\end{lemma}
\begin{proof}
  (a) Since $f$ is of order $e$, we have $\Q(\zeta_e) \subset M_2(\Q(\sqrt{-d}))$ for some square-free positive integer $d.$ Since $\textrm{dim}_{\Q} M_2(\Q(\sqrt{-d}))=8$ and $M_2(\Q(\sqrt{-d}))$ is not a field, we have $\varphi(e) \leq 4.$ Hence, $e \in \{2,3,4,5,6,8,10,12\}.$ It remains to show that $e \ne 5, 10.$ It suffices to prove that $e \ne 5.$ Suppose on the contrary that $e=5.$ Then we have $\Q(\zeta_5) \subset M_2(\Q(\sqrt{-d}))$, and this is a contradiction. Thus, we have $e \ne 5,$ and the result follows. \\
  (b) By (a), we have $|G|=2^m 3^n$ for some non-negative integers $m$ and $n.$ Hence, by Burnside's $p^a q^b$-Theorem, we can see that $G$ is solvable. \\

  This completes the proof.
\end{proof}

In light of Theorem \ref{ordoccurell} and Lemma \ref{sim to Lange lem}, we may use a result of \cite{2} or \cite{4} in the following theorem.

\begin{theorem}\label{powordelli}
  There exists a finite field $k$ and an ordinary elliptic curve $E$ over $k$ such that $G$ is the automorphism group of a polarized abelian surface over $k,$ which is maximal in the isogeny class of $X:=E^2$ if and only if $G$ is one of the following groups (up to isomorphism):
  \begin{center}
  \begin{tabular}{|c|c|}
\hline
$$ & $G$  \\
\hline
$\sharp 1$ & $D_4$  \\
\hline
$\sharp 2$ & $D_{6}$  \\
\hline
$\sharp 3$ & $\textrm{Dic}_{12} $  \\
\hline
$\sharp 4$ & $SL_2(\F_3) $  \\
\hline
$\sharp 5$ & $\Z/12 \Z \rtimes \Z/2\Z  \cong \Z/4\Z \times \textrm{Sym}_3 $  \\
\hline
$\sharp 6$ & $(Q_8 \rtimes \Z/3\Z) \rtimes \Z/2\Z = GL_2 (\F_3)$  \\
\hline
$\sharp 7$ & $ (\Z/6\Z \times \Z/6\Z) \rtimes \Z/2\Z \cong \textrm{Dic}_{12} \rtimes \Z/6\Z$  \\
\hline
$\sharp 8$ & $(Q_8 \rtimes \Z/3\Z) \times \Z/3\Z \cong \mathfrak{T}^* \times \Z/3\Z$  \\
\hline
$\sharp 9$ & $(Q_8 \rtimes \Z/3\Z) \cdot \Z/4\Z \cong \mathfrak{T}^* \rtimes \Z/4\Z$  \\
\hline
\end{tabular}
\vskip 4pt
\textnormal{Table 12}
\end{center}
\end{theorem}

\begin{proof}
Suppose first that there exists a finite field $k$ and an ordinary elliptic curve $E$ over $k$ such that $G$ is the automorphism group of a polarized abelian surface over $k,$ which is maximal in the isogeny class of $X:=E^2.$ In particular, we have $\textrm{End}_k^0(X)=M_2(\Q(\sqrt{-d}))$ for some square-free positive integer $d.$ Thus, by assumption, $G$ is a maximal finite subgroup of $M_2(\Q(\sqrt{-d})),$ and then, by \cite[\S3 and Table 9]{4} or \cite[Theorem 13.4.5]{2}, we can see that $G$ must be one of the 9 groups in the above table. Hence, it suffices to show the converse. We prove the converse by considering them one by one. \\

(1) Take $G=D_4.$ Let $k=\F_{4}.$ Then there is an ordinary elliptic curve $E$ over $k$ such that $\textrm{End}_k^0(E)=\Q(\sqrt{-15})$ by Theorem \ref{isogclass ell}. Let $X=E^2.$ Then we have $D:=\textrm{End}_k^0(X)=M_2(\Q(\sqrt{-15})).$ Let $\mathcal{O}=M_2 \left(\Z \left[\frac{1+\sqrt{-15}}{2} \right] \right).$ By Theorem \ref{mat max}, $\mathcal{O}$ is a maximal $\Z$-order in $D,$ and hence, there exists an abelian surface $X^{\prime}$ over $k$ such that $X^{\prime}$ is $k$-isogenous to $X$ and $\textrm{End}_k(X^{\prime})=\mathcal{O}$ by Proposition \ref{endposs}. It is well-known that $G$ is a maximal finite subgroup of $GL_2(\Q(\sqrt{-15})).$ Hence, $G$ is a maximal finite subgroup of $\mathcal{O}^{\times},$ too. \\

  Now, let $\mathcal{L}$ be an ample line bundle on $X^{\prime}$, and put
  \begin{equation*}
    \mathcal{L}^{\prime}:=\bigotimes_{g \in G} g^* \mathcal{L}.
  \end{equation*}
  Then $\mathcal{L}^{\prime}$ is also an ample line bundle on $X^{\prime}$ that is preserved under the action of $G$ so that $G \leq \textrm{Aut}_k(X^{\prime},\mathcal{L}^{\prime}).$ Since $\textrm{Aut}_k (X^{\prime},\mathcal{L}^{\prime})$ is a finite subgroup of $\textrm{Aut}_k(X^{\prime})=\mathcal{O}^{\times},$ it follows from the maximality of $G$ that $G=\textrm{Aut}_k(X^{\prime},\mathcal{L}^{\prime}).$ \\

  Now, suppose that $Y$ is an abelian surface over $k$ which is $k$-isogenous to $X.$ In particular, we have $\textrm{End}_k^0(Y)=M_2(\Q(\sqrt{-15})).$ Suppose that there is a finite group $H$ such that $H=\textrm{Aut}_k(Y, \mathcal{M})$ for a polarization $\mathcal{M}$ on $Y,$ and $G$ is isomorphic to a proper subgroup of $H.$ Then $H$ is a finite subgroup of $GL_2(\Q(\sqrt{-15})),$ and hence, it follows from the maximality of $G$ as a finite subgroup of $GL_2(\Q(\sqrt{-15}))$ that $H=G$, which is a contradiction. Therefore we can conclude that $G$ is maximal in the isogeny class of $X.$ \\

  (2) Take $G=D_{6}.$ Let $k=\F_4.$ Then a similar argument as in the proof of (1) can be used to show that there is an ordinary elliptic curve $E$ over $k$ such that $G$ is the automorphism group of a polarized abelian surface over $k,$ which is maximal in the isogeny class of $X:=E^2.$ \\

(3) Take $G=\textrm{Dic}_{12}.$ Let $k=\F_{25}.$ Then there is an ordinary elliptic curve $E$ over $k$ such that $\textrm{End}_k^0(E)=\Q(\sqrt{-21})$ by Theorem \ref{isogclass ell}. Let $X=E^2.$ Then we have $D:=\textrm{End}_k^0(X)=M_2(\Q(\sqrt{-21})).$ Let $\mathcal{O}=M_2 (\Z [\sqrt{-21}]).$ By Theorem \ref{mat max}, $\mathcal{O}$ is a maximal $\Z$-order in $D.$ It is also known in \cite{9} that $G$ is a maximal finite subgroup of $GL_2(\Z[\sqrt{-21}])=\mathcal{O}^{\times}.$ Then a similar argument as in the proof of (1) can be used to show that there exists an abelian surface $X^{\prime}$ over $k$ with a polarization $\mathcal{L}^{\prime}$ such that $G=\textrm{Aut}_k(X^{\prime},\mathcal{L}^{\prime})$ and $G$ is maximal in the isogeny class of $X.$ \\

(4) Take $G=SL_2(\F_3).$ Let $k=\F_{25}.$ Then there is an ordinary elliptic curve $E$ over $k$ such that $\textrm{End}_k^0(E)=\Q(\sqrt{-6})$ by Theorem \ref{isogclass ell}. Let $X=E^2.$ Then we have $D:=\textrm{End}_k^0(X)=M_2(\Q(\sqrt{-6})).$ Let $\mathcal{O}=M_2 (\Z [\sqrt{-6}]).$ By Theorem \ref{mat max}, $\mathcal{O}$ is a maximal $\Z$-order in $D.$ It is also known in \cite{9} that $G$ is a maximal finite subgroup of $GL_2(\Z[\sqrt{-6}])=\mathcal{O}^{\times}.$ Then a similar argument as in the proof of (1) can be used to show that there exists an abelian surface $X^{\prime}$ over $k$ with a polarization $\mathcal{L}^{\prime}$ such that $G=\textrm{Aut}_k(X^{\prime},\mathcal{L}^{\prime})$ and $G$ is maximal in the isogeny class of $X.$ \\

  (5) Take $G=\Z/12\Z \rtimes \Z/2\Z.$ Let $k=\F_{5}.$ Then there is an ordinary elliptic curve $E$ over $k$ such that $\textrm{End}_k^0(E)=\Q(\sqrt{-1})$ by Theorem \ref{isogclass ell}. Let $X=E^2.$ Then we have $D:=\textrm{End}_k^0(X)=M_2(\Q(\sqrt{-1})).$ Let $\mathcal{O}=M_2 (\Z[\sqrt{-1}]).$ By Theorem \ref{mat max}, $\mathcal{O}$ is a maximal $\Z$-order in $D.$ Also, by \cite[Table 9]{4} or \cite[Theorem 13.4.5]{2}, $G$ is a maximal finite subgroup of $\mathcal{O}^{\times}=GL_2(\Z(\sqrt{-1})).$ Then a similar argument as in the proof of (1) can be used to show that there exists an abelian surface $X^{\prime}$ over $k$ with a polarization $\mathcal{L}^{\prime}$ such that $G=\textrm{Aut}_k(X^{\prime},\mathcal{L}^{\prime})$ and $G$ is maximal in the isogeny class of $X.$ \\

  (6) Take $G=(Q_8 \rtimes \Z/3\Z) \rtimes \Z/2\Z \cong GL_2(\F_3).$ Let $k=\F_{17}.$ Then there is an ordinary elliptic curve $E$ over $k$ such that $\textrm{End}_k^0(E)=\Q(\sqrt{-2})$ by Theorem \ref{isogclass ell}. Let $X=E^2.$ Then we have $D:=\textrm{End}_k^0(X)=M_2(\Q(\sqrt{-2})).$ Let $\mathcal{O}=M_2 (\Z[\sqrt{-2}]).$ By Theorem \ref{mat max}, $\mathcal{O}$ is a maximal $\Z$-order in $D.$ Then a similar argument as in the proof of (5) can be used to show that there exists an abelian surface $X^{\prime}$ over $k$ with a polarization $\mathcal{L}^{\prime}$ such that $G=\textrm{Aut}_k(X^{\prime},\mathcal{L}^{\prime})$ and $G$ is maximal in the isogeny class of $X.$ \\

  (7) Take $G=(\Z/6\Z \times \Z/6\Z) \rtimes \Z/2\Z.$ Let $k=\F_{7}.$ Then there is an ordinary elliptic curve $E$ over $k$ such that $\textrm{End}_k^0(E)=\Q(\sqrt{-3})$ by Theorem \ref{isogclass ell}. Let $X=E^2.$ Then we have $D:=\textrm{End}_k^0(X)=M_2(\Q(\sqrt{-3})).$ Let $\mathcal{O}=M_2 \left(\Z \left[\frac{1+\sqrt{-3}}{2}\right] \right).$ By Theorem \ref{mat max}, $\mathcal{O}$ is a maximal $\Z$-order in $D.$ Then a similar argument as in the proof of (5) can be used to show that there exists an abelian surface $X^{\prime}$ over $k$ with a polarization $\mathcal{L}^{\prime}$ such that $G=\textrm{Aut}_k(X^{\prime},\mathcal{L}^{\prime})$ and $G$ is maximal in the isogeny class of $X.$ \\

  (8) Take $G=(Q_8 \rtimes \Z/3\Z) \times \Z/3\Z.$ Let $k=\F_7$. Then a similar argument as in the proof of (7) can be used to show that there is an ordinary elliptic curve $E$ over $k$ such that $G$ is the automorphism group of a polarized abelian surface over $k,$ which is maximal in the isogeny class of $X:=E^2.$ \\

  (9) Take $G=(Q_8 \rtimes \Z/3\Z) \cdot \Z/4\Z.$ Let $k=\F_5$. Then a similar argument as in the proof of (5) can be used to show that there is an ordinary elliptic curve $E$ over $k$ such that $G$ is the automorphism group of a polarized abelian surface over $k,$ which is maximal in the isogeny class of $X:=E^2.$ \\

  This completes the proof.
\end{proof}

 Finally,
\begin{theorem}\label{pow of supell}
  There exists a finite field $k$ and a supersingular elliptic curve $E$ over $k$ (all of whose endomorphisms are defined over $k$) such that $G$ is the automorphism group of a polarized abelian surface over $k,$ which is maximal in the isogeny class of $X:=E^2$ if and only if $G$ is one of the following groups (up to isomorphism):
  \begin{center}
  \begin{tabular}{|c|c|}
\hline
$$ & $G$  \\
\hline
$\sharp 1$ & $2_{-}^{1+4}.\textrm{Alt}_5 $  \\
\hline
$\sharp 2$ & $SL_2(\F_3) \times \textrm{Sym}_3 $  \\
\hline
$\sharp 3$ & $(SL_2(\F_3))^2 \rtimes \textrm{Sym}_2 $  \\
\hline
$\sharp 4$ & $SL_2(\F_9)  $  \\
\hline
$\sharp 5$ & $\Z/3\Z : (SL_2(\F_3).2) $  \\
\hline
$\sharp 6$ & $(\textrm{Dic}_{12})^2 \rtimes \textrm{Sym}_2 $  \\
\hline
$\sharp 7$ & $SL_2(\F_5).2  $  \\
\hline
$\sharp 8$ & $ SL_2(\F_5):2 $  \\
\hline
$\sharp 9$ & $(Q_8 \rtimes \Z/3\Z) \rtimes \Z/2\Z \cong GL_2 (\F_3)$  \\
\hline
$\sharp 10$ & $\Z/12 \Z \rtimes \Z/2\Z \cong \Z/4\Z \times \textrm{Sym}_3 $  \\
\hline
$\sharp 11$ & $(Q_8 \rtimes \Z/3\Z) \cdot \Z/4\Z \cong \mathfrak{T}^* \rtimes \Z/4\Z$  \\
\hline
$\sharp 12$ & $ (\Z/6\Z \times \Z/6\Z) \rtimes \Z/2\Z \cong \textrm{Dic}_{12} \rtimes \Z/6\Z$  \\
\hline
$\sharp 13$ & $(Q_8 \rtimes \Z/3\Z) \times \Z/3\Z \cong \mathfrak{T}^* \times \Z/3\Z$  \\
\hline
$\sharp 14$ & $ D_4$  \\
\hline
$\sharp 15$ & $D_6 $  \\
\hline
$\sharp 16$ & $\mathfrak{I}^* \cong SL_2 (\F_5)$ \\
\hline
$\sharp 17$ & $\textrm{Dic}_{24}$  \\
\hline
$\sharp 18$ & $\mathfrak{O}^*$  \\
\hline
$\sharp 19$ & $ \mathfrak{T}^* \cong SL_2 (\F_3)$    \\
\hline
$\sharp 20$ & $\textrm{Dic}_{12}$ \\
\hline
\end{tabular}
\vskip 4pt
\textnormal{Table 13}
\end{center}
\end{theorem}

\begin{proof}
Suppose first that there exists a finite field $k=\F_q$ and a supersingular elliptic curve $E$ over $k$ (all of whose endomorphisms are defined over $k$) such that $G$ is the automorphism group of a polarized abelian surface over $k,$ which is maximal in the isogeny class of $X:=E^2.$ In particular, we have $\textrm{End}_k^0(X)=M_2(D_{p,\infty}).$ Thus, by assumption, $G$ is a maximal finite subgroup of $GL_2(D_{p,\infty})$ (up to isomorphism). If $G$ is absolutely irreducible, then $G$ must be one of the 8 groups ($\sharp 1\sim\sharp 8$) in the above table by Theorem \ref{prim aimf of GL2} and Theorem \ref{imprim aimf GL2}. If $G$ is irreducible (but not absolutely irreducible), then, as we have seen in Section \ref{quat mat rep}, $G$ must be one of the 12 groups ($\sharp 9\sim\sharp 20$) in the above table. If $G$ is reducible, then $G \leq D_{p,\infty}^{\times} \times D_{p,\infty}^{\times},$ and maximality of $G$, together with Goursat's Lemma gives us that $G=G_1 \times G_2$ where $G_1, G_2$ are irreducible maximal finite subgroups of $D_{p,\infty}^{\times}.$ There are 5 such groups, namely, $SL_2(\F_3) \times SL_2(\F_3), \textrm{Dic}_{12} \times \textrm{Dic}_{12}, \Z/6\Z \times \Z/6\Z, \Z/4\Z \times \Z/4\Z,$ and $\Z/2\Z \times \Z/2\Z.$ The first two are contained in $\sharp 3$ and $\sharp 6$, respectively. Also, it can be shown that all the other three groups cannot be maximal in the isogeny class. Hence, it suffices to show the converse. We prove the converse by considering them one by one. \\

  (1) Let $k=\F_{4}.$ Then there is a supersingular elliptic curve $E$ over $k$ such that $\textrm{End}_k^0(E)=D_{2,\infty}$ by Theorem \ref{isogclass ell}. Let $X=E^2.$ Then we have $D:=\textrm{End}_k^0(X)=M_2(D_{2,\infty}).$ Let $V=D_{2,\infty}^2, \mathcal{O} = \Z \left[i, j, ij, \frac{1+i+j+ij}{2} \right],$ and $G=2_{-}^{1+4}.\textrm{Alt}_5.$ By Theorem \ref{prim aimf of GL2}, $G$ is a primitive absolutely irreducible maximal finite subgroup of $GL_2 (D_{2,\infty}).$ Recall also that $\mathcal{O}$ is a maximal $\Z$-order in $D_{2,\infty}.$ By Lemma \ref{main lem hard}, there is a $G$-invariant $\mathcal{O}$-lattice $L$ in $V,$ and then by Theorem \ref{mat max 2}, it follows that $\mathcal{O}^{\prime}:=\textrm{Hom}_{\mathcal{O}}(L,L)$ is a maximal $\Z$-order in $\textrm{Hom}_{D_{2,\infty}}(V,V)=M_2(D_{2,\infty}).$ By the choice of $L,$ $G$ can be regarded as a subgroup of $(\mathcal{O}^{\prime})^{\times}.$ Also, by Proposition \ref{endposs}, there exists an abelian surface $X^{\prime}$ over $k$ such that $X^{\prime}$ is $k$-isogenous to $X$ and $\textrm{End}_k(X^{\prime})=\mathcal{O}^{\prime}$. Then it follows that $G \leq \textrm{Aut}_k(X^{\prime}).$ \\

  Now, let $\mathcal{L}$ be an ample line bundle on $X^{\prime}$, and put
  \begin{equation*}
    \mathcal{L}^{\prime}:=\bigotimes_{g \in G} g^* \mathcal{L}.
  \end{equation*}
  Then $\mathcal{L}^{\prime}$ is also an ample line bundle on $X^{\prime}$ that is preserved under the action of $G$ so that $G \leq \textrm{Aut}_k(X^{\prime},\mathcal{L}^{\prime}).$ Since $\textrm{Aut}_k (X^{\prime},\mathcal{L}^{\prime})$ is a finite subgroup of $\textrm{Aut}_k(X^{\prime})=(\mathcal{O}^{\prime})^{\times},$ it follows from the maximality of $G$ that $G=\textrm{Aut}_k(X^{\prime},\mathcal{L}^{\prime}).$ \\

  Now, suppose that $Y$ is an abelian surface over $k$ which is $k$-isogenous to $X.$ In particular, we have $\textrm{End}_k^0(Y)=M_2(D_{2,\infty}).$ Suppose that there is a finite group $H$ such that $H=\textrm{Aut}_k(Y, \mathcal{M})$ for a polarization $\mathcal{M}$ on $Y,$ and $G$ is isomorphic to a proper subgroup of $H.$ Then $H$ is a finite subgroup of $GL_2(D_{2,\infty}),$ and hence, it follows from the maximality of $G$ as a finite subgroup of $GL_2(D_{2,\infty})$ that $H=G$, which is a contradiction. Therefore we can conclude that $G$ is maximal in the isogeny class of $X.$ \\

(2) Take $G=SL_2(\F_3) \times \textrm{Sym}_3.$ Let $k=\F_{4}$. Then a similar argument as in the proof of (1) can be used to show that there is a supersingular elliptic curve $E$ over $k$ (all of whose endomorphisms are defined over $k$) such that $G$ is the automorphism group of a polarized abelian surface over $k,$ which is maximal in the isogeny class of $X:=E^2.$ \\

(3) Take $G=(SL_2(\F_3))^2 \rtimes \textrm{Sym}_2.$ Let $k=\F_{4}$. Then $G$ is an imprimitive absolutely irreducible maximal finite subgroup of $GL_2(D_{2,\infty})$ by Theorem \ref{imprim aimf GL2}, and hence, a similar argument as in the proof of (1) can be used to show that there is a supersingular elliptic curve $E$ over $k$ (all of whose endomorphisms are defined over $k$) such that $G$ is the automorphism group of a polarized abelian surface over $k,$ which is maximal in the isogeny class of $X:=E^2.$ \\

 (4) Let $k=\F_{9}.$ Then there is a supersingular elliptic curve $E$ over $k$ such that $\textrm{End}_k^0(E)=D_{3,\infty}$ by Theorem \ref{isogclass ell}. Let $X=E^2.$ Then we have $D:=\textrm{End}_k^0(X)=M_2(D_{3,\infty}).$ Let $V=D_{3,\infty}^2, \mathcal{O} = \Z \left[1, i, \frac{1+j}{2}, \frac{i+ij}{2} \right],$ and $G=SL_2(\F_9).$ By Theorem \ref{prim aimf of GL2}, $G$ is a primitive absolutely irreducible maximal finite subgroup of $GL_2 (D_{3,\infty}).$ Recall also that $\mathcal{O}$ is a maximal $\Z$-order in $D_{3,\infty}.$ By Lemma \ref{main lem hard}, there is a $G$-invariant $\mathcal{O}$-lattice $L$ in $V,$ and then by Theorem \ref{mat max 2}, it follows that $\mathcal{O}^{\prime}:=\textrm{Hom}_{\mathcal{O}}(L,L)$ is a maximal $\Z$-order in $\textrm{Hom}_{D_{3,\infty}}(V,V)=M_2(D_{3,\infty}).$ Then a similar argument as in the proof of (1) can be used to show that there exists an abelian surface $X^{\prime}$ over $k$ with a polarization $\mathcal{L}^{\prime}$ such that $G=\textrm{Aut}_k(X^{\prime},\mathcal{L}^{\prime})$ and $G$ is maximal in the isogeny class of $X.$ \\

(5) Take $G=\Z/3\Z : (SL_2(\F_3).2).$ Let $k=\F_{9}$. Then a similar argument as in the proof of (4) can be used to show that there is a supersingular elliptic curve $E$ over $k$ (all of whose endomorphisms are defined over $k$) such that $G$ is the automorphism group of a polarized abelian surface over $k,$ which is maximal in the isogeny class of $X:=E^2.$ \\

(6) Take $G=(\textrm{Dic}_{12})^2 \rtimes \textrm{Sym}_2.$ Let $k=\F_{9}$. Then $G$ is an imprimitive absolutely irreducible maximal finite subgroup of $GL_2(D_{3,\infty})$ by Theorem \ref{imprim aimf GL2}, and hence, a similar argument as in the proof of (4) can be used to show that there is a supersingular elliptic curve $E$ over $k$ (all of whose endomorphisms are defined over $k$) such that $G$ is the automorphism group of a polarized abelian surface over $k,$ which is maximal in the isogeny class of $X:=E^2.$ \\

 (7) Let $k=\F_{25}.$ Then there is a supersingular elliptic curve $E$ over $k$ such that $\textrm{End}_k^0(E)=D_{5,\infty}$ by Theorem \ref{isogclass ell}. Let $X=E^2.$ Then we have $D:=\textrm{End}_k^0(X)=M_2(D_{5,\infty}).$ Let $V=D_{5,\infty}^2, \mathcal{O} = \Z \left[1, i, \frac{1+i+j}{2}, \frac{2+i+ij}{4} \right],$ and $G=SL_2(\F_5).2.$ By Theorem \ref{prim aimf of GL2}, $G$ is a primitive absolutely irreducible maximal finite subgroup of $GL_2 (D_{5,\infty}).$ Recall also that $\mathcal{O}$ is a maximal $\Z$-order in $D_{5,\infty}.$ By Lemma \ref{main lem hard}, there is a $G$-invariant $\mathcal{O}$-lattice $L$ in $V,$ and then by Theorem \ref{mat max 2}, it follows that $\mathcal{O}^{\prime}:=\textrm{Hom}_{\mathcal{O}}(L,L)$ is a maximal $\Z$-order in $\textrm{Hom}_{D_{5,\infty}}(V,V)=M_2(D_{5,\infty}).$ Then a similar argument as in the proof of (1) can be used to show that there exists an abelian surface $X^{\prime}$ over $k$ with a polarization $\mathcal{L}^{\prime}$ such that $G=\textrm{Aut}_k(X^{\prime},\mathcal{L}^{\prime})$ and $G$ is maximal in the isogeny class of $X.$ \\

(8) Take $G=SL_2(\F_5) : 2.$ Let $k=\F_{25}$. Then a similar argument as in the proof of (7) can be used to show that there is a supersingular elliptic curve $E$ over $k$ (all of whose endomorphisms are defined over $k$) such that $G$ is the automorphism group of a polarized abelian surface over $k,$ which is maximal in the isogeny class of $X:=E^2.$ \\

(9) Let $k=\F_{169}$. Note first that $\Q(\sqrt{-2})\subseteq D_{13,\infty}.$ Also, there is a supersingular elliptic curve $E$ over $k$ such that $\textrm{End}_k^0(E)=D_{13,\infty}$ by Theorem \ref{isogclass ell}. Let $X=E^2.$ Then we have $D:=\textrm{End}_k^0(X)=M_2(D_{13,\infty}).$ Let $V=D_{13,\infty}^2, \mathcal{O}$ a maximal $\Z$-order in $D_{13,\infty}$, and $G=(Q_8 \rtimes \Z/3\Z) \rtimes \Z/2\Z = GL_2 (\F_3).$ As we have seen before in Example \ref{GL(2,3)}, $G$ is an irreducible maximal finite subgroup of $GL_2 (D_{13,\infty}).$ By Lemma \ref{main lem hard}, there is a $G$-invariant $\mathcal{O}$-lattice $L$ in $V,$ and then by Theorem \ref{mat max 2}, it follows that $\mathcal{O}^{\prime}:=\textrm{Hom}_{\mathcal{O}}(L,L)$ is a maximal $\Z$-order in $\textrm{Hom}_{D_{13,\infty}}(V,V)=M_2(D_{13,\infty}).$ Then a similar argument as in the proof of (1) can be used to show that there exists an abelian surface $X^{\prime}$ over $k$ with a polarization $\mathcal{L}^{\prime}$ such that $G=\textrm{Aut}_k(X^{\prime},\mathcal{L}^{\prime})$ and $G$ is maximal in the isogeny class of $X.$ \\

(10) Take $G=\Z/12 \rtimes \Z/2\Z \cong \Z/4\Z \times \textrm{Sym}_3.$ Let $k=\F_{49}$. Then a similar argument as in the proof of (9) can be used to show that there is a supersingular elliptic curve $E$ over $k$ (all of whose endomorphisms are defined over $k$) such that $G$ is the automorphism group of a polarized abelian surface over $k,$ which is maximal in the isogeny class of $X:=E^2.$ \\

(11) Take $G=(Q_8 \rtimes \Z/3\Z) \cdot \Z/4\Z \cong \mathfrak{T}^* \rtimes \Z/4\Z.$ Let $k=\F_{49}$. Then a similar argument as in the proof of (9) can be used to show that there is a supersingular elliptic curve $E$ over $k$ (all of whose endomorphisms are defined over $k$) such that $G$ is the automorphism group of a polarized abelian surface over $k,$ which is maximal in the isogeny class of $X:=E^2.$ \\

(12) Take $G=(\Z/6\Z \times \Z/6\Z) \rtimes \Z/2\Z \cong \textrm{Dic}_{12} \rtimes \Z/6\Z.$ Let $k=\F_{121}$. Then a similar argument as in the proof of (9) can be used to show that there is a supersingular elliptic curve $E$ over $k$ (all of whose endomorphisms are defined over $k$) such that $G$ is the automorphism group of a polarized abelian surface over $k,$ which is maximal in the isogeny class of $X:=E^2.$ \\

(13) Take $G=(Q_8 \rtimes \Z/3\Z) \times \Z/3\Z \cong \mathfrak{T}^* \times \Z/3\Z.$ Let $k=\F_{121}$. Then a similar argument as in the proof of (9) can be used to show that there is a supersingular elliptic curve $E$ over $k$ (all of whose endomorphisms are defined over $k$) such that $G$ is the automorphism group of a polarized abelian surface over $k,$ which is maximal in the isogeny class of $X:=E^2.$ \\

(14) Let $k=\F_{58081}$. Clearly, $M_2(\Q) \subset M_2(D_{241,\infty}).$ Also, there is a supersingular elliptic curve $E$ over $k$ such that $\textrm{End}_k^0(E)=D_{241,\infty}$ by Theorem \ref{isogclass ell}. Let $X=E^2.$ Then we have $D:=\textrm{End}_k^0(X)=M_2(D_{241,\infty}).$ Let $V=D_{241,\infty}^2, \mathcal{O}$ a maximal $\Z$-order in $D_{241,\infty}$, and $G=D_4.$ As we have seen before in Lemma \ref{dihe lem}, $G$ is an irreducible maximal finite subgroup of $GL_2 (D_{241,\infty}).$ By Lemma \ref{main lem hard}, there is a $G$-invariant $\mathcal{O}$-lattice $L$ in $V,$ and then by Theorem \ref{mat max 2}, it follows that $\mathcal{O}^{\prime}:=\textrm{Hom}_{\mathcal{O}}(L,L)$ is a maximal $\Z$-order in $\textrm{Hom}_{D_{241,\infty}}(V,V)=M_2(D_{241,\infty}).$ Then a similar argument as in the proof of (1) can be used to show that there exists an abelian surface $X^{\prime}$ over $k$ with a polarization $\mathcal{L}^{\prime}$ such that $G=\textrm{Aut}_k(X^{\prime},\mathcal{L}^{\prime})$ and $G$ is maximal in the isogeny class of $X.$ \\

(15) Take $G=D_6.$ Let $k=\F_{58081}$. Then a similar argument as in the proof of (14) can be used to show that there is a supersingular elliptic curve $E$ over $k$ (all of whose endomorphisms are defined over $k$) such that $G$ is the automorphism group of a polarized abelian surface over $k,$ which is maximal in the isogeny class of $X:=E^2.$ \\

(16) Let $k=\F_{49}$. Then there is a supersingular elliptic curve $E$ over $k$ such that $\textrm{End}_k^0(E)=D_{7,\infty}$ by Theorem \ref{isogclass ell}. Let $X=E^2.$ Then we have $D:=\textrm{End}_k^0(X)=M_2(D_{7,\infty}).$ Let $V=D_{7,\infty}^2, \mathcal{O}$ a maximal $\Z$-order in $D_{7,\infty}$, and $G=\mathfrak{I}^*.$ As we have seen before in Lemma \ref{div alg lem}, $G$ is an irreducible maximal finite subgroup of $GL_2 (D_{7,\infty}).$ By Lemma \ref{main lem hard}, there is a $G$-invariant $\mathcal{O}$-lattice $L$ in $V,$ and then by Theorem \ref{mat max 2}, it follows that $\mathcal{O}^{\prime}:=\textrm{Hom}_{\mathcal{O}}(L,L)$ is a maximal $\Z$-order in $\textrm{Hom}_{D_{7,\infty}}(V,V)=M_2(D_{7,\infty}).$ Then a similar argument as in the proof of (1) can be used to show that there exists an abelian surface $X^{\prime}$ over $k$ with a polarization $\mathcal{L}^{\prime}$ such that $G=\textrm{Aut}_k(X^{\prime},\mathcal{L}^{\prime})$ and $G$ is maximal in the isogeny class of $X.$ \\

(17) Take $G=\textrm{Dic}_{24}.$ Let $k=\F_{49}$. Then a similar argument as in the proof of (16) can be used to show that there is a supersingular elliptic curve $E$ over $k$ (all of whose endomorphisms are defined over $k$) such that $G$ is the automorphism group of a polarized abelian surface over $k,$ which is maximal in the isogeny class of $X:=E^2.$ \\

(18) Let $k=\F_{121}$. Then there is a supersingular elliptic curve $E$ over $k$ such that $\textrm{End}_k^0(E)=D_{11,\infty}$ by Theorem \ref{isogclass ell}. Let $X=E^2.$ Then we have $D:=\textrm{End}_k^0(X)=M_2(D_{11,\infty}).$ Let $V=D_{11,\infty}^2, \mathcal{O}$ a maximal $\Z$-order in $D_{11,\infty}$, and $G=\mathfrak{O}^*.$ As we have seen before in Lemma \ref{div alg lem}, $G$ is an irreducible maximal finite subgroup of $GL_2 (D_{11,\infty}).$ By Lemma \ref{main lem hard}, there is a $G$-invariant $\mathcal{O}$-lattice $L$ in $V,$ and then by Theorem \ref{mat max 2}, it follows that $\mathcal{O}^{\prime}:=\textrm{Hom}_{\mathcal{O}}(L,L)$ is a maximal $\Z$-order in $\textrm{Hom}_{D_{11,\infty}}(V,V)=M_2(D_{11,\infty}).$ Then a similar argument as in the proof of (1) can be used to show that there exists an abelian surface $X^{\prime}$ over $k$ with a polarization $\mathcal{L}^{\prime}$ such that $G=\textrm{Aut}_k(X^{\prime},\mathcal{L}^{\prime})$ and $G$ is maximal in the isogeny class of $X.$ \\

(19) Let $k=\F_{58081}$. Then there is a supersingular elliptic curve $E$ over $k$ such that $\textrm{End}_k^0(E)=D_{241,\infty}$ by Theorem \ref{isogclass ell}. Let $X=E^2.$ Then we have $D:=\textrm{End}_k^0(X)=M_2(D_{241,\infty}).$ Let $V=D_{241,\infty}^2, \mathcal{O}$ a maximal $\Z$-order in $D_{241,\infty}$, and $G=\mathfrak{T}^*.$ As we have seen before in Lemma \ref{div alg lem}, $G$ is an irreducible maximal finite subgroup of $GL_2 (D_{241,\infty}).$ By Lemma \ref{main lem hard}, there is a $G$-invariant $\mathcal{O}$-lattice $L$ in $V,$ and then by Theorem \ref{mat max 2}, it follows that $\mathcal{O}^{\prime}:=\textrm{Hom}_{\mathcal{O}}(L,L)$ is a maximal $\Z$-order in $\textrm{Hom}_{D_{241,\infty}}(V,V)=M_2(D_{241,\infty}).$ Then a similar argument as in the proof of (1) can be used to show that there exists an abelian surface $X^{\prime}$ over $k$ with a polarization $\mathcal{L}^{\prime}$ such that $G=\textrm{Aut}_k(X^{\prime},\mathcal{L}^{\prime})$ and $G$ is maximal in the isogeny class of $X.$ \\

(20) Take $G=\textrm{Dic}_{12}.$ Let $k=\F_{58081}$. Then a similar argument as in the proof of (19) can be used to show that there is a supersingular elliptic curve $E$ over $k$ (all of whose endomorphisms are defined over $k$) such that $G$ is the automorphism group of a polarized abelian surface over $k,$ which is maximal in the isogeny class of $X:=E^2.$ \\

This completes the proof.
\end{proof}

\bibliographystyle{amsalpha}

\end{document}